\documentclass[11pt]{article}

\usepackage{enumerate}
\usepackage{amsmath}
\usepackage{amssymb,latexsym}
\usepackage{amsthm}
\usepackage{comment}
\usepackage{color}
\usepackage{faktor}
\usepackage{graphicx}
\usepackage{amscd}
\usepackage{hyperref}
\usepackage{multicol}
\usepackage{textgreek}
\usepackage{mathrsfs}
\usepackage{mathtools}

\title{$q$-Polymatroids associated with restricted rank-metric codes}
\author{Eimear Byrne, Giovanni Longobardi and Rocco Trombetti}
\date{}

\newcommand{\numberset}{\mathbb}
\newcommand{\N}{\numberset{N}}

\newcommand{\Q}{\numberset{Q}}

\newcommand{\F}{\numberset{F}}

\newcommand{\C}{\mathcal{C}}

\newcommand{\M}{\mathcal{M}}

\newcommand{\X}{\mathcal{X}}

\newcommand{\Fq}{\F_q}

\newcommand{\colsp}{\textnormal{colsp}}

\newcommand{\Tr}{\textnormal{Tr}}

\newcommand{\rk}{\textnormal{rk}}

\newcommand{\maxrk}{\mathrm{maxrk}}
\newcommand{\ee}{\mathbf{e}}

\newcommand{\fqn}{\mathbb{F}_{q^n}}

\DeclareMathOperator{\zero}{\mathbf{0}}

\DeclareMathOperator{\GL}{GL}

\DeclareMathOperator{\Aut}{Aut}

\DeclareMathOperator{\rowsp}{rowsp}

\newtheorem{theorem}{Theorem}[section]
\newtheorem{lemma}[theorem]{Lemma}
\newtheorem{corollary}[theorem]{Corollary}
\newtheorem{proposition}[theorem]{Proposition}

\DeclareMathOperator{\tr}{Tr}

\DeclareMathOperator{\cC}{\mathcal{C}}
\DeclareMathOperator{\cD}{\mathcal{D}}
\DeclareMathOperator{\cS}{\mathcal{S}}

\DeclareMathOperator{\End}{End}
\DeclareMathOperator{\Sym}{\mathrm{Sym}}
\DeclareMathOperator{\Alt}{\mathrm{Alt}}
\DeclareMathOperator{\Her}{\mathrm{Her}}
\DeclareMathOperator{\cX}{\mathcal{X}}

\theoremstyle{definition}
\newtheorem{definition}[theorem]{Definition}

\newtheorem{remark}[theorem]{Remark}

\begin{document}

\maketitle

\begin{abstract}
In this article, we study polymatroids that are representable by means of linear {\it restricted rank-metric codes}, namely, by subspaces of the space of alternating, symmetric, or Hermitian square matrices endowed with the rank metric. More precisely, we characterize the rank function defining these polymatroids and establish sufficient conditions on the relevant parameters under which it is fully determined. We show that there are several differences in compared to the behaviour of $q$-polymatroids of unrestricted matrix codes. 
    
\end{abstract}



\section{Introduction}

While the earliest papers on the theory of rank-metric codes go back some decades (see \cite{de78,gabidulin1985theory,RothMaximumArrayCodes}) the topic remains an active area of research with applications in network coding, distributed storage, and secret sharing schemes \cite{dinesen2025secretsharingrankmetric,kk08,skk08}. As algebraic and combinatorial objects in their own right, they have inspired numerous avenues of research and there remain many open problems; see \cite{bartzsurvey} for an extensive survey on the topic.

In the unrestricted case, a rank-metric code is an $\Fq$-subspace of matrices for which the underlying ambient space has been endowed with the rank metric (the rank distance between a pair of matrices is simply the rank of their difference). In addition to this case, Delsarte also studied codes with restrictions, such as spaces of alternating matrices \cite{delsarte1975}. More recent results in the restricted case include studies of rank-metric codes as symmetric bilinear forms and as subspaces of Hermitian matrices \cite{COSSIDENTE2022105597,DUMAS2011504,longobardi2020automorphism, schmidt_symmetric_2015,schmidt_hermitian_2018,tangzhou}.
We will use the notation $\Alt_{n,q}$ and $\Sym_{n,q}$ to denote the subspaces of $\F_q^{n \times n}$ made up of alternating and symmetric matrices, respectively, and by $\Her_{n,q}$ the $\F_q$-subspaces of Hermitian matrices in $\F_{q^2}^{n \times n}$. We often use $\cX_{n,q}$ to denote any one of the above mentioned families of codes.

Studies of invariants of rank-metric codes have made many appearances in the scientific literature \cite{Kurihara+,ravanticode}. Such invariants are crucial to our understanding of codes and serve to distinguish inequivalent codes. A number of coding theoretic invariants and behaviours are actually determined by an associated object called a $q$-polymatroid. A $q$-polymatroid is a special function defined on a subspace lattice. To make the connection to a rank-metric code, we view a code as an object that is supported on a finite dimensional vector space over $\Fq$. In this paper, most often the support of a codeword will be its column space as a subspace of $\F_{q^\ell}^n$ for $\ell \in \{1,2\}$ and the support of a code is the vector space sum of the supports of its elements. Then a $q$-polymatroid function on the lattice of subspaces of $\Fq^n$ is determined by the dimensions of the shortened subcodes of the given code.

While most $q$-polymatroids have no associated rank-metric code, those that are, are called representable.
The characteristic polynomial and the co-boundary polynomial of a $q$-polymatroid determine the binomial moments and weight enumerator of a code in the representable case via evaluations. MacWilliams duality for rank-metric codes is easily described from these or from the rank-generating polynomial of a $q$-polymatroid \cite{shiromoto19, byrne2021weighted}. Generalized weights are related to the rank function of the flats of its associated $q$-polymatroid \cite{gorla2019rank,GLJ}. They are also defined in terms of the optimal anticodes of $\Fq^{n \times n}$, which are the codes of maximal dimension among those with a fixed maximum rank. These have been characterized in \cite{meshulam1985maximal} as those subspaces of
$\Fq^{n \times n}$ whose support is contained in a fixed subspace of 
$\F_q^n$. In the context of the work presented here, we examine the behaviour of a special class of anticodes in $\cX_{n,q}$, with a focus on the $q$-polymatroids associated with them.
 More generally, we consider $q$-polymatroids that can be represented by codes in $\cX_{n,q}$. In the unrestricted case, the $q$-polymatroid of a maximum rank distance code in $\Fq^{n \times n}$ is completely determined and is (up to scaling) the uniform $q$-matroid. However, as we show here, for a maximum distance code in $\cX_{n,q}$ the rank function of its $q$-polymatroid is not necessarily determined.
 
 This paper is organised as follows. In Section \ref{sec:preliminaries} we collect preliminary notions and results on restricted rank metric codes, determine lower bounds on the minimum distance of their duals and present introductory material on $q$-polymatroids.
In Section \ref{qpolyrestricted}, we prove our main results. 
For a given subspace $U \leq \F_{q^\ell}^n$, we describe the $q$-polymatroid arising from $\cX_{n,q}(U)$, which is the subspace of $\cX_{n,q}$ whose matrices all have support contained in $U$.
Moreover, we compute the maximum rank of the elements in the duals of such spaces and through this obtain a description of the rank function of $q$-polymatroids arising from $d$-codes (those of minimum rank distance $d$). We consider the $q$-polymatroids of a class of extremal codes in $\cX_{n,q}$ called maximum $d$-codes. Using a description of such codes in terms of linearized polynomials, we show that for $d$ close to $n$, the rank functions of their $q$-polymatroids are often determined or nearly determined.
Finally, in Section \eqref{last_section} we establish some results concerning the connection between the dual of a $d$-code and the dual of the relevant $q$-polymatroid resulting from the code. We show that while a relation exists between the dual of the $q$-polymatroid of a code in $\cX_{n,q}$ and the $q$-polymatroid of the dual code in $\cX_{n,q}$, in contrast to the unrestricted case, in general these two objects do not coincide.

\section{Preliminaries}\label{sec:preliminaries}
In this section, we will fix the notation and collect some preliminary results needed for the remainder of this article.

\subsection{Rank-metric codes}\label{sec:codes}

Let $q$ be a prime power and denote by $\F_q$ the finite field with $q$ elements. Let $\F_{q}^{m \times n}$ denote the space of $m \times n$ matrices with entries over $\F_{q}$. This is a metric space if equipped with the \textit{rank distance} defined as 
\begin{equation}
d_{\rk}(A,B) := \rk(A - B),
\end{equation}
where $A,B \in \F_{q}^{m \times n}$.

A  subset $\cC  \subset \F_{q}^{m \times n}$ is called a \textit{rank-metric code}. If $|\cC | >1$, the \textit{minimum distance} of $\cC$ is defined to be
\begin{equation*}
d=d_{\rk}(\cC): = \min_{\underset{M \neq N}{M, N \in \cC}}
d_{\rk}(M,N).  
\end{equation*}
Its \textit{minimum weight} is defined to be
$\textup{w}_{\rk}(\cC):=\min_{\underset{M \neq 0}{M \in \cC}} \rk(M)$ and its \textit{maximum rank} is defined to be
$\maxrk(\cC):=\max_{\underset{M \neq 0}{M \in \cC}} \rk(M).$
The code $\cC$ is said to be \textit{additive} if it is a subgroup of the group $(\F_q^{m \times n}, +)$ and $\F_q$-\textit{linear} if  it is an $\F_q$-subspace of $\F^{m \times n}_q$.  In this case $\dim \cC$ denotes its dimension over $\F_q$. If $\cC$ is additive then its minimum distance and minimum weight coincide.

If $\cC$ is a rank-metric code, then the code $\cC^\top:=\{A^t : A \in \cC\}$, consisting of the transposes of the matrices belonging to $\cC$, is called the \textit{adjoint code} of $\cC$. 

\medskip
\noindent Let $\chi$ be a non-trivial character of the additive group $(\F_{q},+)$ into ${\mathbb C}$. 
We define an inner product
on $\F_{q}^{m \times n}$ into the complex numbers by
\begin{equation*}\label{unrestricted-scalarproduct}
\langle A, B \rangle := \chi(\tr(AB^t)), 
\end{equation*}
for all $A,B \in \F_{q}^{m \times n}$ where $\mathrm{Tr}(\cdot)$ is the trace of a square matrix.
Then, the \textit{Delsarte-dual} of an additive code $\cC$ is defined as 
\begin{equation}\label{def:delsarte_dual}
    \cC^{\perp}=\{N \in \F_{q}^{m \times n} \colon  \langle M , N \rangle =1\textnormal{ for all } M \in \cC\}.
\end{equation}
Generally, for the applications in classical coding theory \cite{kk08,4608992}, given two positive
integers $m, n$ and $1 \leq d \leq \min\{m, n\}$, it is desirable to have codes with mininum distance $d$
as large as possible.
In \cite{de78}, Delsarte proved that the size of each rank-metric code must satisfy an upper
bound, the so-called \textit{Singleton-like bound}. Precisely,
if $\cC \subset \F_q^{m \times n}$ is a code with minimum distance
$d$, then
\begin{equation}\label{eq:sing}
    | \cC| \leq q^{\max\{m,n\}(\min\{m,n\}-d+1)}.
\end{equation}
The code, $\cC$ is called a \textit{maximum rank distance code} if equality holds in (\ref{eq:sing}). This is indicated by saying that $\cC$ is an \textit{MRD code}. The first families of MRD codes were found by Delsarte in \cite{de78},  Gabidulin in \cite{gabidulin1985theory} and Roth in \cite{RothMaximumArrayCodes}.\\
Let $u \in \{1,2,\ldots,m-1\}$ and $\pi_u :
\F_q^{m \times n} \longrightarrow \F_q^{(m-u)\times n}$ be the projection onto the last $m - u$
rows. The \textit{puncturing of a code $\cC \subset \F_q^{m \times n}$ with respect to $A \in \GL(m,q)$ and $ 1 \leq u \leq m-1$} is defined as
\begin{equation*}
\pi_u(A\cC) =\{\pi_u(AC) \,\colon\, C \in \cC\} \subset  \F_q^{(m-u) \times n}.
\end{equation*}
If $m \leq n$ and $\cC$ is an MRD code, then for any $A \in \GL(m,q)$ and for any $1 \leq u \leq m-1$ the punctured
code $\pi_u(A \cC)$ is also MRD, see \cite[Corollary 35]{covering-byrne-rav}.\\
Let $P \in \GL(m, q), Q\in \GL(n, q)$, $R\in \F_q^{m \times n}$ and $\tau \in \Aut(\F_q)$. An \textit{isometry} of the metric space $(\F_{q}^{m \times n},d_{\rk})$ is a map of the type
\begin{equation}\label{eq:generalequivalenbce}
    \Phi_{P,Q,R,\tau}: X \in \F_{q}^{m \times n} \longrightarrow PX^\tau Q + R \in \F_{q}^{m \times n},
\end{equation}
where $X^\tau$ is the matrix obtained by applying $\tau$ to each entry of $X$.
The group of isometries acts transitively on $\F_{q}^{m \times n}$, see \cite[Proposition 3.1]{Wan96}.
Two rank-metric codes $\cC$ and $\cC'$ are \textit{equivalent} if there exists an isometry $\Phi_{P,Q,R,\tau}$ such that  $\cC'=\Phi_{P,Q,R,\tau}(\cC)$
and, if $n=m$, also if
$\cC'=\Phi_{P,Q,R,\tau}(\cC^\top)$;
see for instance
\cite{berger2003isometries, Morrison20147035}. When $\cC$ and $\cC'$ are additive, we may assume that the matrix $R$ in \eqref{eq:generalequivalenbce} is the zero matrix.

\medskip

For our purposes, we will focus on rank-metric codes of squares matrices of order $n$. We will introduce an equivalent setting where these rank-metric codes can be studied. 

Let $1 \leq s \leq n$  be an integer such that $\gcd(s,n)=1$. It is well-known that for any endomorphism $\varphi(x) \in \mathrm{End}_{\F_q}(\F_{q^n})$, there is a unique polynomial 
\begin{equation}\label{lin-pol}
 f=   \sum_{i=0}^{n-1}a_i X^{q^{si}} \in \F_{q^n}[X],
\end{equation}
such that the map $\varphi(x)$ is equal to the endomorphism of $\F_{q^n}$
\begin{equation}
x \in \F_{q^n} \longmapsto   \sum_{i=0}^{n-1}a_i x^{q^{si}} \in \F_{q^n},
\end{equation}
which will be denoted by $f(x)$, see  \cite{Lidl1997}. A polynomial as in \eqref{lin-pol} is called \textit{linearized polynomial} (or a $q^s$-\textit{polynomial}) \textit{with coefficients over} $\F_{q^n}$ and the maximum integer $r$ such that $a_r \neq 0$ is called $q^s$-\emph{degree} of $f$, which is denoted by $\deg_{q^s} f$. The set of roots and the set of values taken by $f(x)$, namely $\ker f$ and $\mathrm{im}f $, are both $\mathbb{F}_q$-subspaces of $\mathbb{F}_{q^n}$. By the \textit{rank} of $ f$, denoted as $\rk\,f$, we mean the dimension of the $ \mathbb{F}_q$-vector space $\mathrm{im} f$.
   The set
$$\tilde{\mathcal{L}}_{n,q,s}[X]= \left \{\sum_{i=0}^{n-1}a_iX^{q^{si}} \colon a_i \in \F_{q^n} \right \}$$
of linearized polynomials with $q^s$-degree at most $n-1$ is a vector space over $\F_{q^n}$ with
respect to the usual sum and scalar multiplication. This algebraic structure, endowed with the composition $\circ$ of polynomials modulo $X^{q^{sn}}-X$, is an algebra over $\F_q$ isomorphic to $\End_{\F_q}(\F_{q^n})$ and hence to $\F_{q}^{n \times n}$.\\ 
Therefore, the notions recalled before can be reformulated in terms of $q^s$-polynomials. Indeed, a rank-metric code can be seen as a subset ${\cal C}$ of $\tilde{\mathcal{L}}_{n,q,s}[X]$ and its \emph{minimum distance}, if $| \cC| > 1$, is 
	\begin{equation*}
		d(\cC)=\min_{\underset{ f,g \in \cC}{ f\neq g}}  \rk (f-g).
	\end{equation*}

The \textit{adjoint polynomial} $f^{\top}$ of $f=\sum_{i=0}a_i X^{q^{si}} \in \tilde{\mathcal{L}}_{n,q,s}[X]$ is defined as 
\begin{equation*}
f^\top=\sum_{i=0}^{n-1}a^{q^{s(n-i)}}_iX^{q^{s(n-i)}}
\end{equation*}
and hence the adjoint code of $\cC$ turns out to be 
$$\cC^\top = \{ f^\top \in \tilde{\mathcal{L}}_{n,q,s}[X]: f \in \cC\}.$$
Let $r$ be a divisor of $n$ and define $\mathrm{Tr}_{q^n/q^r}(x)=\sum^{n/r-1}_{i=0}x^{q^{ir}}$ the \textit{trace} of an element of $\F_{q^n}$ over $\F_{q^t}$.
Then, in this setting, the Delsarte-dual code turns to be 
\begin{equation}
\cC^\perp= \{f  \in \tilde{\mathcal{L}}_{n,q,s}[X] \colon  \langle f, g \rangle = 1,  \forall g \in \cC\}
\end{equation}
where 
$$\langle f, g \rangle:= \chi \left (\mathrm{Tr}_{q^n/q}\left (\sum_{i=0}^{n-1}f_ig_i \right ) \right )$$
with $f=\sum_{i=0}^{n-1}f_iX^{q^{si}},$ $g=\sum_{i=1}^{n-1}g_iX^{q^{si}}$, and $\chi$ is a non-trivial character.

Finally, we can adapt the notion of equivalence as follows. Two  rank-metric codes $\mathcal{C}$ and $\mathcal{C}'$ are said to be {\it equivalent} if there exist two permutation polynomials $g_1,g_2 \in  \tilde{\mathcal{L}}_{n,q,s}[X]$, i.e. the maps $g_1(x)$ and $g_2(x)$ are bijective, $h \in \tilde{\mathcal{L}}_{n,q,s}[X]$ and  $\tau \in {\rm Aut}(\F_{q})$ such that either \[ \mathcal{C}'= \{ g_1 \circ f^{\tau} \circ g_2 + h   \, : \, f \in \mathcal{C} \}\]
or
\[\mathcal{C}'= \{ g_1 \circ f^{\tau} \circ g_2 + h   \, : \, f \in \mathcal{C}^\top \},\] 
where $f^{\tau}$ stands for a polynomial whose coefficients are the image under $\tau$ of those of $f$, for further details and references on this topic, see \cite{S2019}.

\subsection{Restricted rank-metric codes}
Let $\F_{q^\ell}^{n \times n}$ be the set of square matrices of order $n$  with entries over $\F_{q^\ell}$, $\ell \in \{1,2\}$. Also, denote by $\Alt_{n,q},$ and $\Sym_{n,q}$ the subspaces of $\F_q^{n \times n}$ made up of alternating and symmetric matrices, respectively and by $\Her_{n,q}$ the $\F_q$-subspace of Hermitian matrices in $\F_{q^2}^{n \times n}$.
In order to make the discussion as unified as possible, in the following we will write $\cX_{n,q}$ to denote any one of the spaces $\Alt_{n,q}, \Sym_{n,q},$ and $\Her_{n,q}$. Also, we will set $\ell=2$ when dealing with the Hermitian case, while it will be $\ell=1$ otherwise. 

We will be mainly concerned with $d$-\textit{codes} in $\cX_{n,q}$. These are rank-metric codes $\cC$ contained in  $\cX_{n,q}$ of minimum rank distance $d$.

Let $a \in \F^*_q$, $P \in \GL(n, q)$, $R \in \cX_{n,q}$ and $\tau \in \Aut(\F_{q^\ell})$, an \textit{isometry} of $(\cX_{n,q},d_\rk)$  is a map 
\begin{equation}
\Psi^{\circ,\sigma}_{a,P,R,\tau}: X \in \cX_{n,q} \longrightarrow aP(X^\circ)^\tau (P^\sigma)^t + R \in \cX_{n,q},
\end{equation}
where $X \longmapsto X^\circ$ is 
	\begin{itemize}
		\item [$a)$] the identity map, if $\cX_{n,q}=\Alt_{n,q}$ and $n \neq 4$, $\cX_{n,q}=\Sym_{n,q},\Her_{n,q}$; 
		\item [$b)$] either the identity map or the map 
		\begin{equation*}
		\begin{pmatrix}
		0 & x_{12} & x_{13} & x_{14} \\
		-x_{12} & 0 & x_{23} & x_{24} \\ 
		-x_{13} &- x_{23} & 0 & x_{34} \\
		-x_{14} & -x_{24} & -x_{34} & 0 \\
		\end{pmatrix}
		\longmapsto
		\begin{pmatrix}
		0 & x_{12} & x_{13} & x_{23} \\
		-x_{12} & 0 & x_{14} & x_{24} \\ 
		-x_{13} &- x_{14} & 0 & x_{34} \\
		-x_{23} & -x_{24} & -x_{34} & 0 \\
		\end{pmatrix},
		\end{equation*} if $n=4$ and $\cX_{n,q}=\Alt_{n,q}$,
	\end{itemize}
and $P \longmapsto P^\sigma$ is
\begin{itemize}
    \item [$i)$] the identity map, if $\cX_{n,q}=\Alt_{n,q},\Sym_{n,q}$; 
    \item [$ii)$] the involutory automorphism of $\F_{q^2}$ acting on any entries of $P$, if $\cX_{n,q}=\Her_{n,q}$.
\end{itemize}
The group of isometries of $\cX_{n,q}$ acts transitively on it, see \cite{Wan96}. Also, two codes $\cC,\cC' \subset \X_{n,q}$ are said to be \textit{equivalent} if there exists $\Psi^{\circ,\sigma}_{a,P,R,\tau}$ such that $\cC'=\Psi^{\circ,\sigma}_{a,P,R,\tau}(\cC)$.

\begin{definition}
Let $\chi$ be a non-trivial character of the additive group $(\F_{q},+)$. We define an inner product
on $\cX_{n,q}$ into the complex numbers by
\begin{equation*}\label{scalarproduct}
\langle A, B \rangle := \begin{cases}
    \displaystyle\chi\left(\sum_{1\leq i<j\leq n}a_{ij}b_{ij}\right) \textnormal{ if } \cX_{n,q}=\Alt_{n,q}\\
    \chi(\tr(AB^t)) \quad  \quad \quad\textnormal{ otherwise, }
\end{cases}
\end{equation*}
for all $A=(a_{ij})$ and $B=(b_{ij})$ belonging to $ \cX_{n,q}$.\\

\medskip
\noindent If $\cC$ be an additive code of $\cX_{n,q}$, the \textit{dual} of $\cC$ in $\cX_{n,q}$ is the additive code
\begin{equation*}\label{def:dual}
\cC^* = \left  \{B \in \cX_{n,q} \colon \langle A, B \rangle  = 1, \forall  A \in \cC \right \}.
\end{equation*}    
\end{definition}

In the following we will denote by $d^*$ the minimum distance of the dual code $\cC^*$ of a $d$-code $\cC$. It is straightforward to see that for each additive code $\cC$ of $\cX_{n,q}$, one has that 
\begin{equation}\label{size-relation}
|\cC||\cC^*|=|\cX_{n,q}|.
\end{equation}
Also, it is well known that the subspaces of matrices $\Alt_{n,q},$ $\Sym_{n,q} \subset  \F_{q}^{n \times n}$ and $\Her_{n,q} \subset \F_{q^2}^{n \times n}$ are isomorphic to certain subsets of $\tilde{\mathcal{L}}_{n,q^\ell,s}[X]$ where $\ell=1$ when dealing with the alternating and symmetric cases, and $\ell=2$ otherwise.

The following subsections briefly review how this representation may be achieved. 

\subsubsection{\it Alternating \texorpdfstring{$d$}{d}-codes} 
By suitably choosing an $\F_q$-basis of $\F_{q^n}$, the set of alternating matrices $\Alt_{n,q}$ can be identified with the following vector subspace of $\tilde{\mathcal{L}}_{n,q,s}[X]$: 
\begin{equation}\label{pol-alt}
\Alt_{n,q} =\left \{\sum_{i=1}^{n-1}c_iX^{q^{si}} \colon c_{n-i} = -c^{q^{s(n-i)}},\,\, i \in \{1, 2,\ldots, n-1\}\right \}.
\end{equation}

Clearly, $\dim \Alt_{n,q}=\binom{n}{2}$ and it is well known that the rank of each polynomial $f \in \Alt_{n,q}$ is necessarily even.

In this setting, two codes $\cC,\cC' \subset \Alt_{n,q}$ are equivalent if it happens that 
$$\cC'=\{ag \circ f^\tau \circ  (g^\top)+h \colon f \in \cC\},$$ for given $a\in \F^*_q$, $g$ a permutation polynomial in $\tilde{\mathcal{L}}_{n,q,s}[X]$, $\tau \in \Aut(\F_{q})$ and $h \in  \Alt_{n,q}$. If $\cC$ and $\cC'$ are equivalent additive codes, the polynomial $h$ can be chosen equal to the null one.

In \cite[Theorem $4$]{delsarte1975}, Delsarte and Goethals showed that if $\cC \subset \Alt_{n,q}$ is a $d$-code, then 
\begin{equation}\label{alt-bound}
        |\cC| \leq q^{\frac{n(n-1)}{2\left \lfloor \frac{n}{2} \right \rfloor } \left (\left \lfloor \frac{n}{2} \right \rfloor -\frac{d}{2}+1 \right )}.
\end{equation}
Also, in case of equality, the weight distribution of $\cC$ is uniquely determined.

Let $n$ be an odd integer; in \cite[Theorem 7]{delsarte1975} (see also \cite{longobardi2020automorphism}), the following family of maximum alternating $2e$-codes was presented:

\begin{equation}\label{eq:alternatingcode}
\mathcal{A}_{n,2e,s}=  \Biggl \{  \sum_{i=e}^{\frac{n-1}{2}} \biggl ( b_i X^{q^{si}}- (b_i X)^{q^{s(n-i)}} \biggr ) \, : \,\, b_{e},\ldots, b_{\frac{n-1}{2}} \in \F_{q^n}  \Biggr \}.
\end{equation}

In general, it is not an easy task to compute the weight distribution of an alternating code that does not achieve the bound stated in \eqref{alt-bound}. However, similar to results obtained for Hamming metric codes and unrestricted rank-metric codes (see, e.g. \cite[Theorem 4.5]{byravsiam}), partial information on the weight distribution of an alternating code using knowledge of the minimum distance of its dual code has recently been determined in \cite[Proposition 6.1 and Corollary 6.2]{friedlander2023macwilliams}. 

Next, by relying  on \eqref{alt-bound}, we can easily derive a bound on the minimum distance of the dual code $\C^*$ of an alternating code $\cC$. More precisely, we get the following result.

\begin{lemma}\label{bound-dual-dist}
Let $\cC \subset \Alt_{n,q}$ be an additive $d$-code. Then, the dual code $\cC^*$  is a  $d^*$-code with 
$$d^* \leq \min \left \{2\left \lfloor \frac{n}{2} \right  \rfloor,  2\left \lfloor \frac{n}{2} \right  \rfloor -d +4 \right \}.$$
In particular, if $\cC$ is an $\F_q$-linear code and $d=n$ even, then $d^*=2$.
\end{lemma}
\begin{proof}
Let us set $d=2e$ and $d^*=2e^*$. By \eqref{size-relation}, we have that $\log_q\vert \cC \vert  + \log_q \vert \cC^* \vert  = \binom{n}{2} $. Moreover, by \eqref{alt-bound}, we get
\begin{equation}\label{alt-fist-ineq}
\begin{split}
 \log_q \vert  \cC^* \vert  & \geq \frac{n(n-1)}{2}-\frac{n(n-1)}{2 \left \lfloor \frac{n}{2} \right \rfloor } \left (\left \lfloor \frac{n}{2} \right \rfloor-e+1 \right ) = \frac{n(n-1)}{2 \left \lfloor \frac{n}{2} \right \rfloor }(e-1).
\end{split}
  \end{equation}

Then, again by taking into account the upper bound \eqref{alt-bound}, 
\begin{equation}\label{alt-sec-ineq}
 \log_q  \vert \cC ^*  \vert \leq \frac{n(n-1)}{2 \left \lfloor \frac{n}{2} \right \rfloor } \left (\left \lfloor \frac{n}{2} \right \rfloor-e^*+1 \right ) .
\end{equation}
Putting together \eqref{alt-fist-ineq} and \eqref{alt-sec-ineq}, we get the result. Moreover, by \cite[Lemma 3]{quinlan_gow}, if $d=n$, an $\F_q$-linear rank-metric code $\cC \subset \Alt_{n,q}$ has dimension at most $n/2$. Hence, arguing as above, we get $e^*=1$. Hence the result follows.
\end{proof}

 By \cite[Theorem 5]{delsarte1975}, an  additive $d$-code with minimum distance $2 < d=2e<n$ is maximum if and only if its dual code is maximum as well, with minimum distance $d^*= 2 \lfloor \frac{n}{2} \rfloor -d +4$.

\subsubsection{\it Symmetric \texorpdfstring{$d$}{d}-codes}
 
 The set of symmetric matrices with entries over $\F_q$ can be seen as the $\binom{n+1}{2}$-dimensional subspace of $\tilde{\mathcal{L}}_{n,q,s}[X]$
\begin{equation}\label{pol-sym}
\Sym_{n,q}= \left \{\sum_{i=0}^{n-1}c_iX^{q^{si}} \colon c_{n-i}=c^{q^{s(n-i)}},\,\,i \in \{0, 1,\ldots,n-1\}\right\}.
\end{equation}

Also in this case, two codes $\cC,\cC' \subset \Sym_{n,q}$ are equivalent if for a given $a \in \F^*_q$, $g$ permutation polynomial belonging to $\tilde{\mathcal{L}}_{n,q,s}[X]$, $\tau \in \Aut(\F_{q})$ and $h \in  \cX_{n,q}$, it happens that
$\cC'=\{ag \circ f^\tau \circ  (g^\top)+h \colon f \in \cC\}.$

Again,  if $\cC$ and $\cC'$ are equivalent additive codes, then the polynomial $h$ can be chosen to be equal to the null one.

In \cite[Lemma 3.5]{schmidt_symmetric_2015} it is proven that if $\cC \subset \Sym_{n,q}$ is a $d$-code with minimum distance $d=2\delta -1$, then 
\begin{equation} \label{eq:sym_bound_d_odd}
\log_q | \cC |
\leq
\begin{cases}
    n\left( \frac{n+1}{2}-\delta+1 \right )  & \textnormal{for $n$ odd}\\
    (n+1)\left ( \frac{n}{2}-\delta+1 \right )  & \textnormal{for $n$ even}. 
\end{cases}
\end{equation}

Moreover, equality occurs if and only if the dual code $\cC^*$ of $\cC$ has minimum distance $d^* \geq 2t+3$, where $t = \lfloor (n+1)/2\rfloor-\delta$. 

On the other hand, in \cite[Lemma 3.6]{schmidt_symmetric_2015} it is shown that if otherwise 
$d=2\delta$ and $\cC \subset \Sym_{n,q}$ is an additive $d$-code, then 
\begin{equation} \label{eq:sym_bound_d_even}
\log_q |\cC| 
\leq
\begin{cases}
(n+1)\left( \frac{n-1}{2}-\delta+1 \right )  & \textnormal{for $n$ odd}\\
    n\left ( \frac{n}{2}-\delta+1 \right )  & \textnormal{for $n$ even}. 
\end{cases}
\end{equation}

By using these upper bounds, we may prove a bound on the minimum distance of the dual code of a symmetric code.

\begin{lemma} \label{dual-d-odd-sym}
    Let $\cC \subset \Sym_{n,q}$ be an additive $d$-code with $d=2\delta-1$, for some positive integer $\delta > 1$. Then, $d^* \leq n -d +3.$\\
    In particular if $\cC$ is maximal, $\cC^*$ is also maximal with minimum distance
    \begin{equation*}
d^*=
\begin{cases}
    n-d+3 & \textnormal{if $n$ odd},\\
    n-d+2 &  \textnormal{if $n$ even}.
\end{cases}
    \end{equation*}
\end{lemma}
\begin{proof}
Firstly, let us suppose $n$ is odd. By applying \eqref{eq:sym_bound_d_odd} and since $\log_q \vert \cC \vert  + \log_q \vert \cC^* \vert  = \binom{n+1}{2}$, we get
\begin{equation}\label{eq:lower}
    \log_q\vert \cC ^* \vert \geq \binom{n+1}{2}-n \left (\frac{n+1}2-\delta+1 \right )=n(\delta-1).
\end{equation} 
Taking again into account \eqref{eq:sym_bound_d_odd}, from \eqref{eq:lower} we may write
\begin{equation}
      n(\delta-1)\leq \log_q \vert \cC^* \vert  \leq n \left (\frac{n+1}2-\frac{d^*+1}2+1 \right ),
\end{equation}
which leads us to the assertion. 
Now, assume that $n$ is even. By the same argument, we get
 $$(n+1)(\delta-1) \leq \log_q \vert  \cC^* \vert  \leq n\left ( \frac{n}{2}-\frac{d^*}{2}+1 \right )   $$
 and hence $d^* \leq n-d+3$. In the case of a maximum rank distance code $\cC$, providing $n$ is odd by \cite[Lemma 3.5] {schmidt_symmetric_2015}, it is straightforward to see that $d^* \geq n-d+3$ and the equality follows.\\
 On the other hand, when $n$ is even, one has $n-d+2 \leq d^*\leq n-d+3$. However, if $d^* = n-d+3$ the fact that $\C$ is MRD implies that 
 $$\log_q \vert \cC^* \vert =(n+1) \left (\frac{d+1}{2}-1 \right ) > n \left (\frac{n}{2}-\frac{d^*}{2}+1 \right ),$$ which in turn leads to a contradiction. Hence, also in this case, the result follows.
\end{proof}


By using a similar argument, we get at the following result.

\begin{lemma} \label{dual-sym-even}
Let $\cC \subset \Sym_{n,q}$ be an additive $d$-code, with  $d=2 \delta$ and $\delta \geq 1$.   Then $d^* \leq n -d +2$. 
\end{lemma}

In \cite{schmidt_symmetric_2015}, Kai-Uwe Schmidt presented the following class of $\F_q$-linear symmetric codes for any integer $1 \leq d \,\, \leq \,\, n$ such that $n-d$ is even:
\begin{equation} \label{Schimdtcode}
\mathcal{S}_{n,d,s} = \Biggl \{ b_0 X+  \sum_{i=1}^{\frac{n-d}{2}}   \Bigl ( b_i X^{q^{si}}+(b_iX)^{q^{s(n-i)}}  \Bigr ) \, : \,\, b_0, b_1, \ldots, b_{\frac{n-d}{2}} \in \F_{q^n}  \Biggr \}
\end{equation}
where $1 \leq s \leq n$ and $\gcd(s,n)=1$. The set $\mathcal{S}_{n,d,s}$ turns out to be a  maximum rank-metric code  with minimum distance $d$, \cite[Theorem 4.4]{schmidt_symmetric_2015}. 

Moreover, in \cite{longobardi2020automorphism}, when $q$ is odd, $n=2k$ and $\gcd(s,n)=\gcd(k,n)=1$, the following set of $q$-polynomials was presented:

\begin{align*}
\mathcal{S}_{2k,s}(\eta)= \bigg \{ & a_0X + \sum_{j=1}^{k-2} \left (a_jX^{q^{sj}} + (a_{j}X)^{q^{s(2k-j)}} \right )+ \eta bX^{q^{s(k-1)}} +  aX^{q^{sk}} \\
&+\eta^{q^{s(k+1)}}b^{q^s}X^{q^{s(k+1)}}  \,:\, a_0,a_1,...,a_{k-2} \in \F_{q^n} \text{ and } a,b \in \F_{q^k} \bigg\},
\end{align*}
where $\eta \in\F_{q^n}$ such that $\eta^{\frac{q^{n-1}-1}{q-1}}$ is not a square, which gives rise to a maximum $2$-code in $\Sym_{n,q}$. In \cite[Theorem 5.1]{longobardi2020automorphism}, it was proven that $\mathcal{S}$ is not equivalent to $\mathcal{S}_{2k,2,s}$.

Recently in \cite{tangzhou}, the following  examples have been exhibited:

\begin{equation}\label{code-tangzhou}
\begin{aligned}
\mathcal{T}_{2k,s}(\eta) = \bigg \{ & b_0X^{q^{sk}}+b_1X^{q^{s(k-1)}}+(b_1X)^{q^{s(k+1)}}+\eta b_2X^{q^{s(k-2)}} + (\eta b_2X)^{q^{s(k+2)}} \,:\\
&\, b_1 \in \F_{q^n} \text{ and } b_0,b_2 \in \F_{q^k} \bigg\},
\end{aligned}
\end{equation}
where $\eta \in\F_{q^n}$ is a non-square in $\F_{q^n}$. In Section $3$ of \cite{tangzhou}, the authors proved that $\mathcal{T}_{2k,s}(\eta)$ turns to be a maximum $(n-2)$-codes for $k \in \{3,4,5\}$. 

\subsubsection{\it Hermitian \texorpdfstring{$d$}{d}-codes} 

Regarding upper bounds for an additive $d$-code $\cC$ in $\Her_{n,q}$, in \cite[Theorem 1]{schmidt_hermitian_2018}, K.-U. Schmidt proved that if it is additive, then $$\log_q |\cC| \leq n(n-d+1).$$
When $d$ is odd, this upper bound also holds for non-additive $d$-codes, whereas if $\cC$ is additive, the equality is attained if and only if the dual code $\cC^*$ has minimum distance $d^* \geq n - d + 2$.

In \cite[Theorem 6.1]{mschimdt} the following easy example of a maximum Hermitian 2-code was exhibited: 
\begin{equation}\label{miriam-code}
\mathcal{R}=\left \{(a_{ij}) \in \Her_{n,q}  \colon a_{ii}=0 \right \}.
\end{equation}

However, Hermitian matrices with entries over $\F_{q^2}$ can also be be identified with the $q^{2s}$-polynomials with coefficients over $\F_{q^{2n}}$ and such that $\gcd(2n,s)=1$. Indeed, we have that 
\begin{equation}\label{pol-her}
\Her_{n,q}=
\left \{ \sum_{i=0}^{n-1}
c_iX^{q^{2si}} \colon c_{n-i+1} = c_i^{q^{s(2n-2i+1)}},\,\, i \in \{0, 1, 2,\ldots, n-1\} \right \} .
\end{equation}
Note that if $f=\sum_{i=1}^{n-1} c_i x^{q^{2si}}$ belongs to $\Her_{n,q}$ with $n$ odd, then $c_{\frac{n+1}{2}} \in \F_{q^n}$. Clearly, here all the indices $c_i$'s in \eqref{pol-alt},\eqref{pol-sym} and $\eqref{pol-her}$ are taken modulo $n$. 

In this case, two codes $\cC,\cC' \subset \Her_{n,q}$ are equivalent if for a given $a \in \F^*_q$, permutation polynomial $g \in \tilde{\mathcal{L}}_{n,q^2}[X]$, $\tau \in \Aut(\F_{q^{2}})$ and $h \in  \Her_{n,q}$, one has
$\cC'=\{ag \circ f^\tau \circ  (g^\top)^{q^{2n-1}}+h \colon f \in \cC\}.$

By the same argument used in Lemmas \ref{bound-dual-dist} and \ref{dual-sym-even}, we obtain the next result. We leave the details to the reader.

\begin{lemma} \label{dual-d-odd-her}
    Let $\cC \subset \Her_{n,q}$ be an additive $d$-code, $d \geq 2$. Then $d^* \leq n -d +2$. In particular if $d$ is odd and $\cC$ is maximum, then, $d^*=n-d+2$. 
\end{lemma}

In \cite[Section 4]{schmidt_hermitian_2018}, the following two classes of $\F_q$-linear codes in $\Her_{n,q}$, were presented.

Suppose that $n$ and $d$ are integers with opposite parity such that $1 \leq d \leq n-1$ and $s$ an integer such that $\gcd(2n,s)=1$. Then, the set
\begin{equation}\label{eq:hermitiancodeoppositeparity}
	\mathcal{H}_{n,d,s}= \biggl \{\sum_{j=1}^{\frac{n-d+1}{2}} \biggl ( (b_jX)^{q^{s(2n-2j+2)}}+b^{q^{s}}_jX^{q^{2sj}}  \biggr ) : \,\, b_1, b_2, \ldots, b_{\frac{n-d+1}{2}} \in \F_{q^{2n}} \biggr \},
\end{equation} 
is  a maximum $\F_q$-linear Hermitian $d$-code \cite[Theorem 4]{schmidt_hermitian_2018}.

Also, providing $n$ and $d$ are both odd integers such that $1 \leq d \leq n$ and $s$ as above. Then, the set 
\begin{equation}\label{eq:hermitiancodeodd}
\begin{aligned}
	\mathcal{E}_{n,d,s}=\biggl \{ & (b_0\,X)^{q^{s(n+1)}}+ \sum_{j=1}^{(n-d)/2} \biggl ((b_jX)^{q^{s(n+2j+1)}}+b^{q^{s}}_jX^{q^{s(n-2j+1)}}  \biggr ) :\\
    &\,\, b_0 \in \F_{q^n}  \,\,\text{and} \,\, b_1, \ldots, b_{(n-d)/2} \in \F_{q^{2n}}\biggr \}
\end{aligned}
\end{equation}
is a maximum $\F_q$-linear Hermitian $d$-code \cite[Theorem 5]{schmidt_hermitian_2018}.

Finally, in \cite{trombetti2021maximum}, when $q$ and $n$ are odd and $s$ an integer such that the $\gcd(2n,s)=1$, the set of $q$-polynomials   
\begin{equation*}
\begin{aligned}
\mathcal{H} = \bigg \{ & \sum_{i=1}^{\frac{n-3}{2}} \bigg (c_iX^{q^{2si}} + (c_iX)^{q^{2s(n-2i+1)}}\bigg ) +bX^{q^{2s}\frac{n+1}{2}} \\
& + a\gamma X^{q^{2s}\frac{n-1}{2}} + (a\gamma)^{q^{s(n+2)}}X^{q^{2s}\frac{n+3}{2}}
\,:\, c_i \in \F_{q^{2n}} \text{ and } a,b \in \F_{q^n} \bigg\},
\end{aligned}
 \end{equation*}
where $\gamma \in\F_{q^n}$ such that $\gamma^{\frac{q^{2n-1}-1}{q-1}}$ is not a square in $\F_q,$  was presented,  which gives rise to a maximum $2$-code in $\Her_{n,q}$ not equivalent to $\mathcal{H}_{n,d,s}$ and $\mathcal{E}_{n,d,s}$ \cite[Theorem 6.6]{trombetti2021maximum}.

\medskip
\noindent In what follows,  we will study in depth the $q$-polymatroid associated with a given $d$-code $\cC \subset \mathcal{X}_{n,q}$.

\subsection{\texorpdfstring{$q$}{q}-Polymatroids}
\label{q-polymatroid}
Let $E$ be an $\F_q$-vector space and let us denote by $(\mathscr{L}(E), \leq, \vee, \wedge)$  the \textit{subspace lattice} of $E$, i.e. the lattice of $\mathbb{F}_q$-subspaces of $E$, ordered with respect to inclusion. The join $\vee$ of a pair of elements of $\mathscr{L}(E)$ is their sum and the meet $\wedge$ of a pair of subspaces is their intersection.
The minimal element of $\mathscr{L}(E)$ is the zero vector space $\{\zero\}$ and its maximal element is $E$. 
For each $U \in \mathscr{L}(E)$, we write $U^\perp$ to denote the orthogonal complement of $U$ with respect to a fixed non-degenerate bilinear form of $E$. We recall now the definition and some basic properties of $q$-polymatroids, see also \cite{GLJ}, \cite{gorla2019rank} and \cite{shiromoto19}.

\begin{definition}
A \textit{$(q,r)$-(integer) polymatroid} is a pair $\M=(\mathscr{L}(E), \rho)$ for which $r \in \Q_{>0}$ and $\rho$ is a function $\rho: \mathscr{L}(E) \longrightarrow {\mathbb N}_0$ satisfying the following axioms: 
\begin{itemize}
	\item[(R1)] $0\leq \rho(U) \leq r \cdot \dim U$, for all $U \in \mathscr{L}(E)$.
	\item[(R2)] $U\leq V \Rightarrow \rho(U)\leq \rho(V)$,  for all $U,V \in \mathscr{L}(E)$.
	\item[(R3)] $\rho(U \vee V)+\rho(U\wedge V)\leq \rho(U) + \rho(V)$, for all $U,V  \in \mathscr{L}(E)$.
\end{itemize}
\end{definition}

Let $\mathcal{M}_1=(\mathscr{L}(E),\rho_1)$ and $\mathcal{M}_2=(\mathscr{L}(E),\rho_2)$ be $(q,r)$-polymatroids. They are said to be \textit{equivalent} if there exists an $\F_q$-linear automorphism of $E$ such that $\rho_1(U)=\rho_2(\varphi(U))$ for all $U \in \mathscr{L}(E)$. We will denote this fact in the remainder by the symbol $\mathcal{M}_1 \cong \mathcal{M}_2$.

Let $\mathcal{M}=(\mathscr{L}(E),\rho)$ be a $(q,r)$-polymatroid and let $$[X,Y]=\{ T \in \mathscr{L}(E) \colon X \leq T \leq Y \}$$ be an interval of~$E$. Then  $\M([X,Y])=([X,Y],\rho_{[X,Y]})$ is a $(q,r)$-polymatroid where 
$\rho_{[X,Y]}: [X,Y] \longrightarrow \N_0$ is the map defined by $$\rho_{[X,Y]}(T):= \rho(T)-\rho(X)$$
for every $T \in [X,Y]$. We say that $\M([X,Y])$ is a {\it minor} of $\M$. In the literature, it is  usual to denote by 
\begin{enumerate}
    \item 
     $\M|_Y:=\M([\textbf{0},Y])$, the {\it restriction} of $\M$ to $Y$,
   \item 
   $\M/X:=\M([X,E])$,  {\it contraction} of $\M$ {\it by} $X$ and  
  \item  $\M \setminus T:=\M |_{T^\perp}$ for $T \in \mathscr{L}(E)$, called the {\it deletion} of $T$ {\it from}~$\mathcal{M}$.
\end{enumerate}
If it is not necessary to specify $r$, we simply refer to a $(q,r)$-polymatroid as to a $q$-\textit{polymatroid}. We define a \textit{$q$-matroid} to be a $(q,1)$-polymatroid. 

\bigskip

 In \cite{GLJ, gorla2019rank,shiromoto19}, it is shown how any rank-metric code induces two $q$-polymatroids. In this regard, for a matrix $M\in\F_{q}^{m\times n}$ we denote by $\colsp(M)$ and $\rowsp(M)$ the column-space and the row-space of $M$ over~$\F_q$, respectively. These are subspaces contained in the vector spaces $\F^m_{q}$ and $\F^n_q$, respectively.

 Let $\mathcal{C} \subset \mathbb{F}^{m \times n}_{q}$ be an $\F_q$-linear rank-metric code.
	   For each subspace 
	   $U \leq \F_q^m$ and $V \leq \F_q^n$ respectively, we define
	   $$\mathcal{C}(U,c):=\{A \in \cC : \colsp(A) \leq U\} \, \text{ and } \, \mathcal{C}(V,r):=\{A \in \cC : \rowsp(A) \leq V\},$$ 
and it is straightforward to see that if $U,W \leq \F_q^m$, then
\begin{equation}\label{eq:relations}
\cC(U,c)+\cC(W,c) \subseteq \cC(U+W,c)
\quad \textnormal{and} \quad
\cC(U \cap W,c) = \cC(U,c) \cap \cC(W,c).
\end{equation}
Similarly, if $V,T \leq \F_{q}^n$, then
\[
\cC(V,r)+\cC(T,r) \subseteq \cC(V+T,r)
\quad \textnormal{and} \quad
\cC(V \cap T,r) = \cC(V,r) \cap \cC(T,r).
\]

Now  consider $E=\F_q^m$ and $\rho_c: \mathscr{L} (E) \longrightarrow \N_{0}$  defined by
	   	$$\displaystyle \rho_c(U):=\dim \cC-\dim \cC(U^{\perp},c),$$
        where $U^{\perp}$ is the orthogonal complement of $U$ with respect the  standard inner product of $E$.
	   	In \cite[Theorem 5.2]{gorla2019rank} the authors proved that $(\mathscr{L}(E),\rho_c)$ is a $(q,n)$-polymatroid which we denote in the following by $\M_c[\cC]$. Similarly, we may define a rank function $\rho_r$ in the following way: $$\rho_r(V)=\dim \cC-\dim \cC(V^\perp,r)$$
for all $V \leq E$, where $E=\F_q^n$. This map is also shown to provide a $(q,m)$-polymatroid which is denoted by $\mathcal{M}_r[\cC]$. However, for our purposes, it will be enough to consider only the column $q$-polymatroid $\M_c[\cC]$.


\section{\texorpdfstring{$q$}{q}-Polymatroids of restricted rank-metric codes}
\label{qpolyrestricted}
From now on, we fix $E = \mathbb{F}_{q^\ell}^n$, with $\ell \in \{1,2\}$. We set $E$ to be $\mathbb{F}_{q^2}^n$ when dealing with the Hermitian case, and to be $\mathbb{F}_q^n$ for the alternating or symmetric case.
We denote by $\mathscr{L}(E)$ the lattice of $\mathbb{F}_{q^\ell}$-subspaces of $E$.

\begin{lemma}\label{begin1}
Let $\cC$ be a code in $\cX_{n,q}$ and let 
\begin{equation}\label{sigma}
\sigma:(\alpha_1\dots,\alpha_n) \in E\longrightarrow (\alpha_1^q\dots,\alpha_n^q) \in E.
\end{equation}
Then for any $\F_{q^\ell}$-subspace $U \leq E$, $\ell \in \{1,2\}$, we have that
$$\cC(U^\sigma,r)=\cC(U,c).$$
\end{lemma}
\begin{proof}
\noindent 
The result is trivial for the codes in $\Alt_{n,q}$ and $\Sym_{n,q}$ because in both cases $\sigma$ is simply the identity map.
Let $\cC \subset \Her_{n,q}$ and $M \in \cC$. For any $i,j \in \{1,2,\ldots,n\}$, let $M_i$ denote the $i$-th row of $M$ and let $M^j$ denote the $j$-th column of $M$.  Since $M \in \Her_{n,q}$,

\begin{equation}
\begin{split}
 \cC(U^\sigma,r)&=\{M \in \cC \colon  \langle M_1,M_2,\ldots M_n \rangle_{\F_{q^2}} \leq U^\sigma\}\\
 &=\{M \in \cC \colon  \langle M^\sigma_1,\ldots,M^\sigma_n \rangle_{\F_{q^2}} \leq U\}\\
  &=\{M \in \cC \colon  \langle M^1,\ldots,M^n\rangle_{\F_{q^2}} \leq U\}=\cC(U,c)\\
 \end{split}
 \end{equation}
\end{proof}

\begin{remark}\label{remark} Let $\cC$ be an $\F_q$-linear $d$-code of $\cX_{n,q}$. By Lemma~\ref{begin1}, if $\cC$ is alternating or symmetric, then $\mathcal{M}_c[\cC] = \mathcal{M}_r[\cC]$ and they are both $(q,n)$-polymatroid. 

On the other hand, if $\cC \subset \Her_{n,q}$, it is  an $\F_q$-subspace of $\F_{q^\ell}^{n \times n}$ and using arguments similar to those used in \cite[Theorem 4.3]{gorla2019rank}, one easily shows that  $\mathcal{M}_c[\cC] = (\mathscr{L}(E), \rho_c)$ and $\mathcal{M}_r[\cC] = (\mathscr{L}(E), \rho_r)$ are $(q^2,n)$-polymatroids.
In this case Lemma \ref{begin1}, implies that
\[
\rho_c(U) = \rho_r(U^\sigma)
\quad \text{for all } U \leq \F_{q^2}^n,
\]
hence, the $(q^2,n)$-polymatroids associated with the columns and with the rows are not equivalent in the sense of the definition given in Subsection~\ref{q-polymatroid}.
Indeed, the automorphism $\sigma$ in~\eqref{sigma} is an $\F_q$-linear automorphism of $\F_{q^2}^n$.
\end{remark}
\medskip

With these observations in mind, in the remainder of the article, for any code $\cC$ in $\cX_{n,q}$, we will consider only the $q$-polymatroid $\mathcal{M}_c[\cC]$, which we denote by $\mathcal{M}[\cC]$.

\begin{lemma}\label{lem:equiv}
For any two $\F_{q^\ell}$-subspaces $U,V$ of $E$ of the same dimension, $\cX_{n,q}(U)$ and $\cX_{n,q}(V)$ are equivalent.  
\end{lemma}

\begin{proof}
    Let $U,V \leq E$ both have the same dimension.
    Since $\GL(n,q^\ell)$, $\ell \in \{1,2\}$, acts transitively on the subspaces of $E$ of the same dimension, there exists an automorphism of $E$ that maps $U$ onto $V$. 
    Let $G  \in  \mathrm{GL}(n,q^\ell)$ be the matrix associated to this automorphism with respect to the canonical basis of $E$. Since $G$ is invertible, the map $M \in \cX_{n,q} \mapsto GM(G^{\sigma})^t \in \cX_{n,q}$ is a rank-preserving automorphism of $\cX_{n,q}$ that maps $\cX_{n,q}(U)$ onto $\cX_{n,q}(V)$.
\end{proof}

The shortening $\cX_{n,q}(U)$ of the code $\cX_{n,q}$, defined by the subspace $U \leq E,$ will also play a crucial role in the following. Firstly. it is straightforward to see that $\cX_{n,q}(U)$ is an $\F_q$-vector subspace of $\cX_{n,q}$. Moreover, we have the following

\begin{proposition}\label{trivial}
    Let $U$ be an $u$-dimensional $\F_{q^\ell}$-subspace of $E$. Then

    \begin{equation*}
        \dim \cX_{n,q}(U)=\begin{cases}
            \binom{u}{2}  & \text{ if } \cX_{n,q}=\Alt_{n,q}\\
             \binom{u+1}{2} & \text{ if } \cX_{n,q}=\Sym_{n,q} \\
              u^2     & \text{ if }\cX_{n,q}=\Her_{n,q}
        \end{cases}
    \end{equation*}
    
\end{proposition}
\begin{proof} 
The statement holds trivially true for $U= \{\textbf{0}\}$. Let $u \in \mathbb{Z}^+$ and consider $V=\langle \ee_1,\ee_2\ldots,\ee_u \rangle_{\F_{q^\ell}}$ where $\ee_1,\ee_2\ldots,\ee_u$ are the first $u$ elements of the standard basis of $E$. It is easy to see that 
   \begin{equation*}
   \cX_{n,q}(V)= \left \{\begin{pmatrix}
       A & O_{u \times (n-u)} \\
       O_{(n-u) \times u} & O_{(n-u) \times (n-u)}
   \end{pmatrix} \colon A \in \cX_{u,q} \right \}.
   \end{equation*}
Clearly, in this case we have  $|\cX_{n,q}(V)|=|\cX_{u,q}|$ and so the statement holds for $V$. By Lemma \ref{lem:equiv},  we have that $|\cX_{n,q}(U)|=|\cX_{n,q}(V)|$ for any $u$-dimensional vector space $U \leq E$ and so the result follows.
\end{proof}

Lemma \ref{lem:equiv} and Proposition \ref{trivial} extend \cite[Lemma 61]{Dualityofcodessupported} and \cite[Lemma 4.3]{gianira}.\\
It is well known that equivalent codes have equivalent duals: $\cC_1, \cC_2 \subset \cX_{n,q}$ are equivalent if and only if  $\cC_1^*$ and $\cC_2^*$ are equivalent.


\noindent Let $V = \langle \ee_1,\dots,\ee_{v} \rangle_{\F_{q^\ell}}$, it is straightforward to check that
\begin{equation*}
   \cX_{n,q}(V)^*= 
   \left \{
   \begin{pmatrix}
      O_{v \times v}& A \\
         \bar{A}^t  & B
   \end{pmatrix} 
   \colon A \in \F_{q^\ell}^{v \times (n-v)},B \in \cX_{n-v,q} \right\},
   \end{equation*}
   where 
   \begin{equation}\label{unifying-bar}
     \bar{A}:=
     \begin{cases}
         -A & \textnormal{ if } \cX_{n,q}=\Alt_{n,q}\\
         A & \textnormal{ if } \cX_{n,q}=\Sym_{n,q}\\
         A^\sigma & \textnormal{ if } \cX_{n,q}=\Her_{n,q}
     \end{cases}.
   \end{equation}
   Then we have that $\dim \cX_{n,q}(V)^* =  \dim \cX_{n,q} - \dim \cX_{n,q}(V)$, hence it is an immediate consequence of Lemma \ref{lem:equiv} and Proposition \ref{trivial} that for any $u$-dimensional $\F_{q^\ell}$-subspace $U\leq E$, we have: 
    \begin{equation}\label{anticode-dual}
        \dim\cX_{n,q}(U)^*=\begin{cases}
         \frac{(n-u)(n+u-1)}{2}  & \text{if} \cX_{n,q}=\Alt_{n,q}\\
          \frac{(n-u)(n+u+1)}{2} & \text{if} \cX_{n,q}=\Sym_{n,q} \\
         (n-u)(n+u)  & \text{if}\cX_{n,q}=\Her_{n,q}
        \end{cases}.
    \end{equation}
\medskip
\noindent Along the lines of the works \cite{gorla2019rank,gorla} in which the $q$-polymatroids associated with optimal anticodes were studied, we now describe the $q$-polymatroid $\mathcal{M}[\cX_{n,q}(V)]$ where $V$ is a fixed subspace of $E$.  Note that by Lemma \ref{lem:equiv}, its rank function depends only on the dimension of $V$.
\begin{theorem}
Let $V$ be a $v$-dimensional $\F_{q^\ell}$-vector subspace of $E$, $\ell \in \{1,2\}$. The rank function $\rho$ of the 
$q$-polymatroid $\mathcal{M}[\cX_{n,q}(V)]$ is given by
\begin{equation}
    \rho(U)=
    \begin{cases}
        \binom{v}{2}-\binom{\dim_{\F_q}(V \cap U^\perp)}{2} & \text{if } \cX_{n,q} = \Alt_{n,q},\\[4pt]
         \binom{v+1}{2}-\binom{\dim_{\F_q}(V \cap U^\perp)+1}{2}  & \text{if } \cX_{n,q} = \Sym_{n,q},\\[4pt]
        (v-\dim_{\F_{q^2}}(V \cap U^\perp))(v+\dim_{\F_{q^2}}(V \cap U^\perp))  & \text{if } \cX_{n,q} = \Her_{n,q},
    \end{cases}
\end{equation}
where $U \in \mathscr{L}(E)$.
\end{theorem}

\begin{proof}
    The result follows immediately from Proposition \ref{trivial}
and by the definition of the rank function of the (column) $q$-polymatroid associated with a rank-metric code. \end{proof}

Let $\cC$ be an additive code in $\cX_{n,q}$ and let $U \in \mathscr{L}(E)$. Then, since $\cC(U)= \cC \cap \cX_{n,q}(U)$, one has

\begin{equation}\label{sizeformula}
    |\cC(U)|=\frac{|\cX_{n,q}(U)| |\cC|}{|\cX_{n,q}(U) + \cC|}
    =\frac{|\cX_{n,q}(U)||\cC^* \cap \cX_{n,q}(U)^*|}
    {|\cC^*|}.
\end{equation}

We can apply \eqref{sizeformula} to analyze the rank function $\rho$ of the $q$-polymatroids $\mathcal{M}[\cC]$ associated with certain classes of restricted $\F_q$-linear codes. \\

Note that if $\cC=\cX_{n,q}$, the $q$-polymatroid is completely determined. Similarly, if $\cC$ is an $\F_q$-linear code with $d=n$ (which, in the case of an alternating code, implies that $d$ is even), then the $q$-polymatroid is also completely determined by its parameters. In this case, the rank function of $\mathcal{M}[\cC]$ is $\rho(U)=\dim \cC$ for any non-null subspace  $ U \in \mathscr{L}(E)$.

\begin{proposition} \label{polymatroid}
    Let $\cC \subset \cX_{n,q}$  be an $\F_q$-linear $d$-code. 
    Then the rank function of $\mathcal{M}[\cC]=(\mathscr{L}(E),\rho)$ satisfies the following: 
    \begin{equation} \label{rank-function}
    \rho(U)= 
    \begin{cases}
        \dim \cC & \textnormal{if } u > n-d \\
        \dim \cX_{n,q}(U^\perp)^* & \textnormal{if } \max\rk \cX_{n,q}(U^\perp)^* < d^*,
    \end{cases}     
    \end{equation}
    for any $u$-dimensional subspace $U \leq E$.
\end{proposition}

\begin{proof}
   By definition, we have $\rho(U) = \dim\cC-\dim\cC(U^\perp)$.
   If $n-u < d$ then $\cC(U^\perp)=\{\textbf{0}\}$ and so $\rho(U)=\dim(\cC)$.
   Using (\ref{sizeformula}) and taking into account (\ref{size-relation}), we have
   \begin{equation*}
    |\cC(U^\perp)|=\frac{|\cC||\cC^* \cap \cX_{n,q}(U^\perp)^*|}{|\cX_{n,q}(U^\perp)^*|}.
\end{equation*}
   If $\max\rk \cX_{n,q}(U^\perp)^* < d^*$, we have $\cC^* \cap \cX_{n,q}(U^\perp)^*=\{\textbf{0}\}$ and hence 
   $$\rho(U) = \dim \cC - (\dim(\cC)-\dim\cX_{n,q}(U^\perp)^*),$$ 
   from which the result follows.
\end{proof}

We now compute the value of $\max\rk \cX_{n,q}(U)^*$ for an $u$-dimensional vector space $U$ of $E$. In order to do it, we will make use of the notion of {\it generalized inverse}; see \cite{FULTON197823} and \cite{ matsaglia_styan} and the references therein. For the sake of completeness, we provide here the definition
\begin{definition}
    Let $A \in \F_{q}^{m \times n}$. A \textit{generalized inverse} or a \textit{g-inverse} of $A$ is any matrix $A^- \in \F_{q}^{n \times m}$ such that $A A^- A = A$, 
\end{definition}
As consequence of \cite[Theorem 3.1]{FULTON197823}, we have that any matrix has a generalized inverse.

\begin{proposition}\label{maxrank} Let $U$ be an $u$-dimensional subspace of $E$. Then 
\begin{equation*} 
\max\rk \cX_{n,q}(U)^*= 
\begin{cases}
\max \rk \cX_{n,q} & \text{if} \quad  u \leq \lfloor \frac{n}{2} \rfloor \\
2(n-u) & \text{if} \quad  u \geq \lceil \frac{n}{2} \rceil
\end{cases}
\end{equation*}
\end{proposition}
\begin{proof}
If $U=\{\textbf{0}\}$, then $\cX_{n,q}(U)^*=\cX_{n,q}$ and hence 
$\max\rk \cX_{n,q}(U)^*= \max \rk \cX_{n,q}$.
By Lemma \ref{lem:equiv}, it is enough to show the statement for $U=\langle \ee_1,\ldots,\ee_u\rangle$ where $\ee_i$ is the $i$-th canonical vector of $E$. In such a case, any matrix in $\cX_{n,q}(U)^*$ has the following form:
\begin{equation*}
C=
\begin{pmatrix}{}
O_{u \times u}  & A \\
\bar{A}^t & B
\end{pmatrix},
\end{equation*}
where $O_{u \times u}$ is the zero matrix in  $\F_{q^\ell}^{u \times u}$, $A \in \F_{q^\ell}^{u \times (n-u)}$ and $B \in \cX_{n-u,q}$. 
By \cite[Theorem 19(8.3) and (2.32)]{matsaglia_styan}, we have that:
\begin{align*}
\rk(C) & = 2 \rk(A) + \rk((I_{n-u}-\bar{A}^t(\bar{A}^t)^-)B(I_{n-u}-A^- A))  \\
& \leq 2 \rk(A)+ \min \{\rk(B), n-u-\rk(A)\}\\
& = \min\{ 2\rk(A)+\rk(B),n-u+\rk(A)\}.
\end{align*}


Let $u \leq \lfloor \frac{n}{2} \rfloor$ and consider the following choice of matrix $C$ in $\cX_{n,q}(U)^*$:
\begin{equation*}
C= \begin{pmatrix}
O_{u\times u} & I_u & O_{u \times (n-2u)} \\
\bar{I}_u & O_{u \times u} & O_{u \times (n-2u)} \\
O_{(n-2u) \times u} & O_{(n-2u) \times u } 
 & B'
  \end{pmatrix}
    \end{equation*}
where $B' \in \cX_{n-2u,q}$ with maximum rank.
  Then $\rk(C) = 2u + \rk(B') = 2u + \max \rk \cX_{n-2u,q} = \max \rk \cX_{n,q}$, which attains 
  the maximal value in $\cX_{n,q}$ and hence in $\cX_{n,q}(U)^*$.\\
Suppose now that $u \geq \lceil \frac{n}{2} \rceil$, in which case we have
$0 \leq \rk(A) \leq n-u$ and $0 \leq \rk(B) \leq  \max \rk \cX_{n-u,q}$.

Then, the maximum value of
$\min\{n-u+\rk(A),2\rk(A)+\rk(B)\},$
for $A$ and $B$ varying in $\F_{q^\ell}^{u \times (n-u)}$ and $\cX_{n-u,q},$ respectively, is $2(n-u)$. It is attained by choosing, for instance, the block matrix with $B=O_{(n-u) \times (n-u)}$ and $A$ with rank equal to $n-u$.
\end{proof}

As a direct consequence of the above proposition and of Proposition \eqref{polymatroid}, we have that if $\cC \subset \cX_{n,q}$ is an $\F_q$-linear $d$-code and $\rho$ is the rank function of the associated $q$-polymatroid $\mathcal{M}[\cC]$, then
\begin{equation} \label{rk-f-restricted}
    \rho(U)= 
    \begin{cases}
        \dim \cC & \textnormal{if } u > n-d \\
        \dim \cX_{n,q} (U^\perp)^* & \textnormal{if }  \min \{2u, \maxrk \cX_{n,q} \}< d^* .
    \end{cases}     
    \end{equation}
where $d^*$ is the minimum distance of $\cC^*$.

\medskip

Finally, in \cite[Theorem 5.5]{GLJ} the relation between the column $q$-polymatroid of a code $\cC$ and the one associated to the code obtained by puncturing 
$\cC$ is described. More precisely, let $\cC \subset \F_q^{m \times n}$ be an $\F_q$-linear code and $u \in \{1,\ldots,m-1\}$, then
\begin{equation}\label{punct.vs.del.}
    \mathcal{M}_c[\pi_u(N\cC)] \cong \mathcal{M}_c[\cC] \setminus T
\end{equation}
where $N=\begin{pmatrix} B \\ D \end{pmatrix} \in \GL(m,q)$, $D \in \F_{q}^{(m-u) \times m}$ and $T=\rowsp(D)^\perp$.

As a simple byproduct of the result above, we have the following statement.

\begin{proposition}
Let $\cC \subset \F_q^{m \times n}$ be an $\F_q$-linear code and let $u \in \{1,\ldots,m-1\}$. Then
\begin{equation}
\mathcal{M}_c[\pi_u(\cC)] \cong \mathcal{M}_c[\cC] \setminus U,
\end{equation}
where $U = \langle \ee_1, \ee_2, \ldots, \ee_u \rangle \leq \F_q^n$ and $\ee_i$ denotes the $i$-th vector of the standard basis of $\F_q^n$.
\end{proposition}
\begin{proof}
    The result follows from the fact that for $N=I_m$ in \eqref{punct.vs.del.},
    \begin{equation*}
    D=
    \begin{pmatrix}
    O_{(m-u) \times u} & I_{m-u}        
    \end{pmatrix}
    \end{equation*}
and $U:=\rowsp(D)^\perp=\langle \ee_1,\ldots,\ee_u \rangle$.
\end{proof}

Let $1 \le u < n$ and consider \(V \in \mathscr{L}(\F_q^u)\) and
\(W \in \mathscr{L}(\F_q^n)\).
We denote by \(V^{\perp_u}\) and \(W^{\perp_n}\) the orthogonal complements
of \(V\) and \(W\) with respect to the standard inner products on
\(\F_q^u\) and on \(\F_q^n\), respectively.
Then the following result holds.

\begin{lemma}\label{orthogonal}
Let $1 \le u < n$ and let $
\psi : \mathbf{v} \in \F_q^{u} \longrightarrow  \mathbf{v}A \in\F_q^n, $
where $A = \left( I_u \;\; O_{u \times (n-u)} \right)$. 
Then, for any $V \in \mathscr{L}(\F_q^u)$,
\[
\psi(V)^{\perp_n} = \psi(V^{\perp_u}) \oplus U^{\perp_n},
\]
where $U = \operatorname{rowsp}(A)$.
\end{lemma}

\begin{proof}
By definition, the map $\psi$ embeds a vector 
$\mathbf v = (v_1, \dots, v_u) \in \F_q^u$ into $\F_q^n$ by adding $n-u$ zeros coordinates. 
Hence, $U = \psi(\F_q^u) = \operatorname{rowsp}(A)$.\\
Firstly, we shall prove that $\psi(V^{\perp_u}) + U^{\perp_n} \subseteq \psi(V)^{\perp_n}$. Since $\psi(V) \subseteq U$, we immediately have 
$U^{\perp_n} \subseteq \psi(V)^{\perp_n}$.  

Next, let $\mathbf x \in \psi(V^{\perp_u})$. Then there exists 
$\mathbf z \in V^{\perp_u}$ such that $\mathbf x  = \mathbf z A$.  
Similarly, for any $\mathbf y \in \psi(V)$, $\mathbf y = \psi(\mathbf v) = \mathbf v A$ with $\mathbf v \in V$.  
Then
\[
\mathbf x \cdot \mathbf y = (\mathbf z A) \cdot (\mathbf v A) = \mathbf z A A^t \mathbf v^t = \mathbf z \mathbf v^t = \mathbf z \cdot \mathbf v = 0,
\]
since $\mathbf z \in V^{\perp_u}$.  
Hence, $\mathbf x \in \psi(V)^{\perp_n}$ and therefore  $\psi(V^{\perp_u}) + U^{\perp_n} \subseteq \psi(V)^{\perp_n}.$
Since $\psi$ is injective, we have
\begin{equation*}
\dim \psi(V^{\perp_u}) + \dim U^{\perp_n} = (u - \dim V) + (n - u) = n - \dim V = \dim \psi(V)^{\perp_n}.
\end{equation*}
Finally, since $\psi(V^{\perp_u}) \cap U^{\perp_n}  \subseteq U \cap U^{\perp_n}= \{\mathbf{0}\}$, the sum is direct and  this concludes the proof.
\end{proof}

    Let $\cC \subset \cX_{n,q}$. Then,
 for each $u \in \{1,\ldots,n-1\}$, we denote by
    $\cC^{[u]} \subset \cX_{u,q}$ the code obtained from $\cC$ by deleting the last $n-u$ rows and the last $n-u$ columns from each element of $\cC$. Note that $\cC^{[u]}=A \cC A^t$ where $A=\left ( I_u \,\, O_{u \times (n-u)}  \right )$.

\begin{proposition} \label{prop:punctured}
    Let $\cC \subset \cX_{n,q}$ be an $\F_q$-linear code and $u \in \{1,2,\ldots,n-1\}$. Then
    $$\mathcal{M}[\cC^{[u]}] \cong 
         \mathcal{M}_c[\cC A^t]|_U $$
    where $A= \left ( I_u \,\, O_{u \times (n-u)} \right )$ and $U=\rowsp(A)$.
\end{proposition}

\begin{proof} Let us recall that $\mathcal{M}_c[\cC]= ( \mathscr{L}(E), \rho_c)$ and  $\mathcal{M}_c[\cC A^t]|_U =(\mathscr{L}(U),\rho')$ where $\rho'(T)=\dim (\cC A^t)-\dim(\cC A^t)(T^{\perp_n},c)$ for any $T \leq U$. Let us consider the $(q^\ell,u)$-polymatroid $\mathcal{M}_c[\cC^{[u]}]=(\mathscr{L}(\F_{q^\ell}^{u}),\tilde{\rho}_c)$. By definition, we have 
$$\tilde{\rho}_c(V) = \dim\cC^{[u]}-\dim\cC^{[u]}(V^{\perp_u}),$$
for all $V \leq \F_{q^\ell}^u$.

Consider the map $ \phi:  CA^t  \in \cC A^t \longmapsto  ACA^t \in \cC^{[u]}$. It is straightforward to see that $\cC^{[u]}$ is isomorphic to $(\cC A^t) /(\cC A^t)(U^{\perp_n},c)$ and so that $\dim\cC^{[u]}=\dim(\cC A^t)-\dim(\cC A^t)(U^{\perp_n},c)$.\\
Now, let $V$ be an $\F_{q^\ell}$-subspace of $\F_{q^\ell}^u$ and set $W=\{ \mathbf{v}A : \mathbf{v} \in V \}$, then
\begin{equation}
\begin{aligned}
\cC^{[u]}(V^\perp)&=\{ A C A^t \in \cC^{[u]} \colon \mathbf{v} A C A^t= \mathbf{0},\, \forall \mathbf{v} \in V \}\\
&=\phi( \{ C A^t \in \cC A^t \colon \mathbf{w} C A^t= \mathbf{0}, \forall \mathbf{w} \in W\})=\phi((\cC A^t)(W^{\perp_n},c)).
\end{aligned}
\end{equation}
By Lemma \ref{orthogonal}, $W^{\perp_n}= \psi(V^{\perp_u}) \oplus U^{\perp_n}$ and hence $ (\cC A^t)(U^{\perp_n},c) \subset (\cC A^t) (W^{\perp_n}, c)$. Then, $ (\cC A^t)(W^{\perp_n},c) / \cC A^t(U^{\perp_n},c) \cong  \cC^{[u]}(V^{\perp_u}) $ and
\begin{equation}
\begin{aligned}
    \tilde{\rho}_c(V)& =\dim \cC^{[u]} - \dim \cC^{[u]} (V^{\perp_u})=\dim (\cC A^t) - \dim (\cC A^t)(U^{\perp_n},c)\\
    &-\dim ((\cC A^t)(W^{\perp_n}, c)) + \dim (\cC A^t)(U^\perp,c)\\
    &= \rho'(W).
    \end{aligned}
\end{equation}
    Clearly, there is a lattice isomorphism between the subspace lattice $\mathscr{L}(\F_{q^{\ell}}^u)$ and $\mathscr{L}(U)$ via $V \mapsto \{\mathbf{v}A : \mathbf{v} \in V\}$. Hence, the result follows.   
\end{proof}

\begin{remark}
 The interaction of the subspaces $\cX_{n,q}(U)$ with an $\Fq$-linear code $\cC\subset \cX_{n,q}$ yields invariants similar to the generalized weights of unrestricted rank-metric codes, as defined in \cite{ravanticode}.
%
%
    Let $\cC \subset \cX_{n,q}$ be an $\F_q$-linear code of dimension $k$. For each $1 \leq j \leq n$ we define the $j$-th $\cX_{n,q}$-weight of $\cC$ to be
$$d_j(\cC,\cX_{n,q})= \min\{ \dim_{\Fq}(\cX_{n,q}(U)): U \in \mathscr{L}(E), \dim(\cC(U)) \geq j\}.$$

It is easy to see that the $d_j(\cC,\cX_{n,q})$ form a non-decreasing sequence.
By Proposition \ref{trivial}, if $\textup{d}_{\rk}(\cC)=d$, then 
$d_1(\cC,\cX_{n,q})$ is bounded from below by $\binom{d}{2}$, $\binom{d+1}{2}$, or $d^2$ for $\cX_{n,q}$ equal to $\Alt_{n,q}$, $\Sym_{n,q}$, or $\Her_{n,q}$ respectively.
\end{remark}

\subsection{\texorpdfstring{$q$}{q}-Polymatroids of alternating codes}

In this subsection, we will derive some results about the $q$-polymatroid of an alternating $d$-code.

\begin{proposition}
Let $\cC \subset \Alt_{n,q}$, $n \geq 3$, be an alternating $\F_q$-linear $d$-code, $ 2 \leq d < n$. 
Then $\mathcal{M}[\cC]$ is fully determined by the parameters $n,d$ in the following cases:
\begin{enumerate}[(a)]
    \item $n \geq 6$ even and $d=n-2$; in which case we have

        \begin{equation*}
    \rho(U)= \begin{cases}
        \dim \cC & \textnormal{ if } u > 2, \\
       \frac{u(2n-u-1)}{2} & \text{ otherwise }       
    \end{cases}.
    \end{equation*}
    
    \item $n \geq 5$ odd and $d=n-1$; in which case we have 
\begin{equation*}
\rho(U)= 
\begin{cases}
\dim \cC & \text{if } u \ge 2\\
u(n-1) & \text{otherwise}
\end{cases}.
\end{equation*}
\end{enumerate}
\end{proposition}

\begin{proof} Let $U$ be an $u$-dimensional subspace of $\F_{q}^n$.
Case $(a)$ directly follows from \eqref{rk-f-restricted}, and by taking into account Lemma \ref{bound-dual-dist}.
    Indeed, if $n \geq 6$ and $d= n-2$,
    \begin{equation}
    \rho(U)= \begin{cases}
        \dim \cC & \textnormal{ if } u > 2 \\
        \dim \Alt_{n,q}(U^\perp)^* & \textnormal{ if } u \leq 2        
    \end{cases}.
    \end{equation}
The same can be said for Case $(b)$ when $n\geq5$. Again, by \eqref{rk-f-restricted}
 and Lemma \ref{bound-dual-dist}
     \begin{equation}
    \rho(U)= \begin{cases}
        \dim \cC & \textnormal{ if } u > 1 \\
        \dim \Alt_{n,q}(U^\perp)^* & \textnormal{ if } u=1        
    \end{cases},
    \end{equation}
   and taking into account \eqref{anticode-dual}, the result follows.
\end{proof}

Let $n=2e+3$ and now consider the maximum rank distance code $\mathcal{A}_{n,2e,s}$ in class \eqref{eq:alternatingcode}.  We have the following

\begin{theorem}  \label{alternating-missing}
    Let $\mathcal{M}[\mathcal{A}_{n,2e,s}] = (\mathscr{L}(E),\rho)$ be the $q$-polymatroid associated with the $2e$-code $\mathcal{A}_{n,2e,s},$ with $n=2e+3$.
    Then, for any $u$-dimensional subspace $U \leq  E$, we have 
    \begin{equation*}
            \rho(U)=\left\{
            \begin{array}{cl}
                2n                           &  \textit{if }   u >3, \\
               \frac{u(2n-u-1)}{2} & \textit{if}\,\, u< 3 ,
            \end{array}
            \right .
        \end{equation*}
and $\rho(U) \in \{2n,2n-1\}$, if $u=3$.
\end{theorem}
 \begin{proof}Firstly note that, up to isomorphism, any $\F_q$-subspace of $\F^n_{q}$ can be seen as an $\F_{q}$-subspace of $\F_{q^{n}}$. Then, let $V\leq \F_{q^{n }}$ be an $\F_{q}$-subspace of $\F_{q^n}$ and consider the $\F_{q}$-linear map
\begin{equation}\label{restriction}
\phi_V : f \in \tilde{\mathcal{L}}_{n,q,s}[X] \longrightarrow f_{|_V}(x) \in \mathrm{Hom}(V,\F_{q^n})
\end{equation}
    where $f_{|_V}(x)$ denotes the restriction of $f(x)$ to the subspace $V$. Since $(\ker f)^\perp=\text{im}f^{\top},$ it is easy to see that, if $\cC \subset \tilde{\mathcal{L}}_{n,q,s}[X]$, the rank function $\rho_c$ of the $q$-polymatroid $\mathcal{M}[\cC]$ can be written as $\rho(V) = \dim_{\F_q} \phi_V(\cC)$. By \eqref{rk-f-restricted}, Proposition \ref{polymatroid}  and \cite[Theorem 5]{delsarte1975}, the value of the rank function $\rho$ is determined for every $u$-dimensional subspace of $\F_{q^n}$ with $u \neq 3$.\\
Let $\mathcal{A}' =X^{q^{s(e+3)}} \circ \mathcal{A}_{2e+3,2e,s} $, which is given by:
    \begin{equation*}
      \mathcal{A}' = \{ a^{q^{se}} X - a X^{q^{3s}} + b^{q^{s(e+1)}} X^{q^s} - b X^{q^{2s}} : a,b \in \fqn\}.
      \end{equation*}
Observe that $\mathcal{M}[\mathcal{A}']$ is equivalent to $\mathcal{M}[\mathcal{A}_{2e+3,2e,s}]$, see \cite[Proposition 6.7]{gorla2019rank}.
Let $U$ be a $3$-dimensional $\F_q$-subspace of $\F_{q^n}$. For any $f \in \tilde{\mathcal{L}}_{n,q,s}[X]$, we have that $f \in \ker \phi_U$ if and only if $U \leq \ker f$. Since the $q^s$-degree of $f \in \mathcal{A}'$ is at most $3$, $\dim (\ker f) \leq 3$. Therefore, since $u=3$, $f \in \ker \phi_U$ if and only if $\ker f = U$.
    Furthermore, if $f_1,f_2 \in \tilde{\mathcal{L}}_{n,q,s}[X]$ both have $q^s$-degree $3$ and satisfy $U \leq \ker f_1 \cap \ker f_2$, then 
    $\langle f_1 \rangle_{\fqn} = \langle f_2 \rangle_{\fqn}$, since otherwise there would exist
 $a_1,a_2 \in \fqn$ such that $a_1f_1+a_2f_2$ has $q^s$-degree at most 2, which brings us to a contradiction. Since 
    $\langle f \rangle_{\fqn} \cap \mathcal{A}' = \langle f \rangle_{\Fq} \cap \mathcal{A}' $, it follows that
    $\mathcal{A}' \cap \ker \phi_U$ has $\Fq$-dimension at most $1$.
    Hence, we have that 
    $$\rho(U) =  \dim \phi_U(\mathcal{A}')=\dim \mathcal{A}'-\dim (\ker \phi_U \cap \mathcal{A}') \in \{2n, 2n-1\}.$$
    This leads to the result.
\end{proof}

\subsection{\texorpdfstring{$q$}{q}-Polymatroids of symmetric codes}

Following the same approach as above, we derive some results on the 
$q$-polymatroids associated with symmetric $d$-codes.

\begin{proposition}
    Let $\cC \subset \Sym_{n,q}$, be a symmetric $\F_q$-linear $d$-code with $1 < d < n$.  Then $\mathcal{M}[\cC]$ is fully determined by the parameters $n,d$, in the following cases:
    \begin{itemize}
        \item [$a)$] $n \geq 3$ and  $d=n-1$;
        \item [$b)$] $n \geq 5$ odd and $d=n-2$.
    \end{itemize}
\end{proposition}
\begin{proof} Let $U$ be a $u$-dimensional subspace of $\F_{q}^n$. Then, by \eqref{rk-f-restricted} we have
    \begin{equation} \label{rank-function-sym}
    \rho(U)= 
    \begin{cases}
        \dim \cC & \textnormal{if } u > n-d \\
        \dim \Sym_{n,q} (U^\perp)^* & \textnormal{if }  \min \{2u,  n  \}< d^* 
    \end{cases}     
    \end{equation}
    for any $u$-dimensional subspace $U \leq \F_{q}^n$.
The statement follows  by Lemma \ref{dual-d-odd-sym} and Lemma \ref{dual-sym-even}. Indeed for $d=n-1$, we get
\begin{equation*}
    \rho(U)=
    \begin{cases}
        \dim \cC & \textnormal{if $u \geq 2$}\\
        nu  & \textnormal{otherwise}
    \end{cases},
\end{equation*}
and this concludes Case $(a)$. Similarly, if $n$ is odd $d=n-2$, we have that
\begin{equation*}
    \rho(U)=
    \begin{cases}
        \dim \cC & \textnormal{if $u \geq 3$}\\
        \frac{u(2n-u+1)}{2}  & \textnormal{if $u \leq 2$}
    \end{cases}
\end{equation*}
and this concludes the proof.
\end{proof}

In the next result, we shall determine the $q$-polymatroid associated with  $\mathcal{S}_{n,d,s}$ described in \eqref{Schimdtcode} in the case when $n-d=2$. We will split our discussion into two cases depending on whether $n$ is even or odd.

\begin{theorem}
    Let $n,d$ be integers with $n \geq 4$ even, $n-d=2$ and consider $\mathcal{M}[\mathcal{S}_{n,d,s}]$. Then, for any non-null $u$-dimensional subspace $U$ of  $\F_q^n$, 
\begin{equation*}
    \rho(U)=
    \left\{
            \begin{array}{cl}
                2n                            &  \textit{if }   u >2 \\
               n & \textit{if}\,\, u=1 ,
            \end{array}
            \right .
        \end{equation*}
and $\rho(U) \in \{2n, 2n-1\}$ if $u=2$.
\end{theorem}
\begin{proof}
 By \eqref{rk-f-restricted} and Lemma \ref{dual-sym-even},  the rank function $\rho$ is determined for any $\F_q$-subspace whose dimension is different from 2.
Let
    \begin{equation*}
       \mathcal{S}'=X^{q^s} \circ \mathcal{S}_{n,n-2,s}= \{ a X + b X^{q^s} + a^{q^s} X^{q^{2s}} : a,b \in \fqn\}.
      \end{equation*}
Note that the polynomials in $\mathcal{S}'$ have $q^s$-degree at most $2$. Then, the result follows by applying an argument similar to that used in the proof of Theorem \ref{alternating-missing} and, therefore, we omit here the details.
\end{proof}

Similarly, the following result can be established.

\begin{theorem}
    Let $n,d$ be integers with $n \geq 7$ odd, $n-d=4$ and consider $\mathcal{M}[\mathcal{S}_{n,d,s}]$. Then, for any $u$-dimensional subspace $U$ of $\F_q^n$,
\begin{equation*}
    \rho(U)=
    \left\{
            \begin{array}{cl}
                3n                            &  \textit{if }   u >4, \\
               \frac{u(2n-u+1)}{2} & \textit{if}\,\, u<4 ,
            \end{array}
            \right .
        \end{equation*}
        and $\rho(U) \in \{3n, 3n-1\}$, if $u=4$.
\end{theorem}
\begin{proof}
The result follows by considering the code $\mathcal{S}'=X^{q^{2s}} \circ \mathcal{S}_{n,n-4,s}$ and applying again same techniques as in Theorem  \ref{alternating-missing}.
\end{proof}

In addition, the following statement holds.

\begin{theorem} 
    Let $n=2k$ and $k\in \{3,4,5\}$ and consider $\mathcal{M}[\mathcal{T}_{2k,s}(\eta)]$ with $\mathcal{T}_{2k,s}(\eta)$ as in \eqref{code-tangzhou}.
    Then, for any $u$-dimensional subspace $U$ of  $\F_q^n$, we have 
    \begin{equation*}
            \rho(U)=\left\{
            \begin{array}{cl}
                2n                            &  \textit{if }   u >2, \\
               \frac{u(2n-u-1)}{2} & \textit{if}\,\, u< 2,
            \end{array}
            \right .
        \end{equation*}
and $\rho(U) \in \{2n,2n-1,2n-2\}$, if $u=2$.
\end{theorem}

\begin{proof}
By \eqref{rk-f-restricted} and Lemma \ref{dual-sym-even}, the rank function  $\rho$ of the q-polymatroid $\mathcal{M}[\mathcal{T}_{2k,s}(\eta)]$ is determined for any subspace of dimension $u \neq 2$.\\
 Consider the code
    $\mathcal{T}' =X^{q^{s(k+2)}} \circ \mathcal{T}_{2k,s}(\eta) $, which is given by:
\begin{align*}
      \mathcal{T}' = & \{ \eta^{q^{s(k+2)}}b_2^{q^{2s}}X + b_1^{q^{s(k+2)}} X^{q^s} +b_0^{q^{2s}}X^{q^{2s}}+ (b_1 X)^{q^{3s}}+(\eta b_2X)^{q^{4s}}\\
      &: b_0,b_2 \in \F_{q^k} \text{ and } b_1 \in \F_{q^{n}}\}.
\end{align*}
Observe that $\mathcal{M}[\mathcal{T}']$ is equivalent to $\mathcal{M}[\mathcal{T}_{2k,s}(\eta)]$, see \cite[Proposition 6.7]{gorla2019rank}.
 By hypotheses, each $f \in \mathcal{T}'$ has kernel of dimension at most $2$. Therefore, let $U$ be a $2$-dimensional $\F_q$-subspace of $\F_{q^n}$. Then, $f \in \ker \phi_U$ if and only if $\ker f = U$, where here the map $\phi_U$ is defined as in \eqref{restriction}.
 Hence, if $f= \sum_{i=0}^{4}f_i X^{q^{si}} \in \mathcal{T}' \cap \ker \phi_U $, then \begin{equation}\label{eq:U_in_kernel_f} f= g \circ m_U, \end{equation} where $g=\sum_{i=0}^{n-1}g_iX^{q^{si}}$ and $m_U=a_0X+a_1X^{q^s}+X^{q^{2s}}$ is the minimal polynomial of $U$. By equating the coefficients in both sides of \eqref{eq:U_in_kernel_f}, one easily gets that the $g_i $'s form a solution of a linear system of $n$ equation in $n$ unknowns with rank $n-2$.
    Hence, we have that 
    $$\rho(U) =  \dim \phi_U(\mathcal{T}')=\dim \mathcal{T}'-\dim (\ker \phi_U \cap \mathcal{T}') \in \{2n, 2n-1,2n-2\},$$
    which gives the result.
\end{proof}

Also in \cite{schmidt_symmetric_2015}, the author showed the following.

\begin{theorem}\label{th:punctured}  \cite[Theorem 4.1]{schmidt_symmetric_2015}.
Suppose that $\cC$ is a maximal additive $d$-code in $\Sym_{n,q}$ for
some $d \geq 3$ such that $n - d -1$ is even. Then $\cC^{[n-1]}$ is a maximal additive code in $\Sym_{n-1,q}$ with minimum distance $d-2$.
\end{theorem}

Then, let us consider the maximum $(d+2)$-code $\mathcal{S}_{n+1,d+2,s}$ in $\Sym_{n+1,q}$ with $n-d-1$ even. Then $\cS'_{n,d,s}=\cS_{n+1,d+2,s}^{[n]} \subset \Sym_{n,q}$ is also maximum with minimum distance $d$.

\begin{proposition} Let $\cC \subset \Sym_{n,q}$ be an $\F_q$-linear maximal $d$-code such that $n-d-1$ is even. Then,  
\begin{equation}
\dim \cC^*(U)=\begin{cases} 
\binom{n}{2} + d-1 & if \, \mathrm{diag}(0,0,\ldots,0,1) \in \cC\\
\binom{n}{2} + d-2 & otherwise
\end{cases}
\end{equation}
where $U=\langle \mathbf{e}_1,\mathbf{e}_2,\ldots,\mathbf{e}_{n-1} \rangle$, with $\mathbf{e}_1,\ldots,\mathbf{e}_{n-1}$ the $i$-th vector of the standard basis of $\F_q^n$.
\end{proposition}
\begin{proof}
By Proposition~\ref{prop:punctured}, the $q$-polymatroid $\mathcal{M}[\cC^{[n-1]}]=\bigl(\mathscr{L}(\F_q^{\,n-1}),\rho_1\bigr)$ and $
\mathcal{M}[\cC A^t]\big|_{U}=\bigl(\mathscr{L}(U),\rho_2\bigr)$ are equivalent where \(A=(I_u \; O_{u\times (n-u)})\) and $U=\rowsp(A)$. The equivalence is given by the linear isomorphism 
$$\psi:(x_1,x_2,\ldots,x_{n-1})\in \F_q^{\,n-1}\longrightarrow (x_1,x_2,\ldots,x_{n-1},0) \in \F_{q}^n.$$ By Theorem \ref{th:punctured}, $\cC^{[n-1]} \subset \Sym_{n-1,q}$ is a maximal $(d-2)$-code. Then, by Formula \eqref{rk-f-restricted},
 \begin{equation*}
\dim \cC^{[n-1]}=\dim(\cC A^t)-\dim (\cC A^t) (\langle \mathbf{e}_n \rangle, c)). 
 \end{equation*}
It is straightforward to check that $\cC / \cC(U^\perp) \cong \cC A^t $ and $\cC /(\cC \cap \Sym_{n,q}(U)^*) \cong (\cC A^t)(\langle \mathbf{e}_n \rangle,c)$, respectively.
Then, 
\begin{equation}
\begin{aligned}
\frac{(n-1)(n-d+3)}{2} & = \dim (\cC \cap \Sym_{n,q}(U)^*)-\dim \cC(U^\perp)\\
&=\dim \Sym_{n,q}- \dim(\cC \cap \Sym_{n,q}(U)^*)^*-\dim \cC(U^\perp)\\
&=\dim \Sym_{n,q}- \dim(\cC^* + \Sym_{n,q}(U))-\dim\cC(U^\perp)\\
&=\dim \cC +\dim\cC^*(U)-\dim \Sym_{n,q}(U)- \dim \cC(U^\perp).
\end{aligned}
\end{equation}
Hence,
\begin{equation}
\begin{aligned}
    \dim \cC^*(U)&= \frac{(n-1)(n-d+3)}{2}-\frac{(n+1)(n-d+1)}{2}+ \binom{n}{2}+\dim \cC(U^\perp)\\
    &= \binom{n}{2}+ d-2 + \dim \cC(U^\perp)
    \end{aligned}
\end{equation}
Since the dimension of $\cC(U^\perp)$ is either one or zero, depending on whether the matrix $\mathrm{diag}(0,0,\ldots,0,1)$ belongs to $\cC$ or not, we get the result.
\end{proof}

\subsection{\texorpdfstring{$q^2$}{q2}-Polymatroids of Hermitian codes}

\begin{proposition}
    Let $\cC \subset \Her_{n,q}$, be an $\F_q$-linear $d$-code, $d\geq 2$.  If $n \geq 3$ and $d=n-1$, then the $q^2$-polymatroid $\mathcal{M}[\cC]$ is fully determined by the parameters.
\end{proposition}\begin{proof}  By Formula \eqref{rk-f-restricted}, the rank function $\rho$ of the $q^2$-polymatroid $\mathcal{M}[\cC]$ in   can be written as
    \begin{equation} \label{rank-function-her}
    \rho(U)= 
    \begin{cases}
        \dim \cC & \textnormal{if } u > n-d \\
        \dim \Her_{n,q} (U^\perp)^* & \textnormal{if }  \min \{2u,  n  \}< d^* .
    \end{cases}     
    \end{equation}
    for any $u$-dimensional $\F_{q^2}$-subspace $U$ of $\F_{q^2}^n$.
By Lemma \ref{dual-d-odd-her}, indeed for $d=n-1$, we get
\begin{equation*}
    \rho(U)=
    \begin{cases}
        \dim \cC & \textnormal{if $u \geq 1$}\\
        u(n-1)^2  & \textnormal{otherwise}
    \end{cases}
\end{equation*}
 and this proves the claim.
\end{proof}

\begin{proposition}
Let $\mathcal{R}=\{(a_{ij}) \in \Her_{n,q} \colon a_{ii}=0\}$ be the maximal $2$-code of $\Her_{n,q}$  as defined in \eqref{miriam-code}. Then the rank function of $\mathcal{M}[\mathcal{R}]$ 
\[
\rho(U)=u(2 n - u-1),
\]
where $U$ is a $u$-dimensional $\F_{q^2}$-subspace of $\F_{q^2}^n$.
\end{proposition}
\begin{proof}
    The claim follows directly from the definition of rank function for a rank-metric code.
\end{proof}

\begin{theorem}\label{her-code-opposite}
    Let $n,d$ be integers with $n \geq 4$ even, $n-d=3$ and  consider $\mathcal{M}[\mathcal{H}_{n,d,s}]$ with $\mathcal{H}_{n,d,s}$ as in \eqref{eq:hermitiancodeoppositeparity}. Then, for any  $u$-dimensional $\F_{q^2}$-subspace $U$ of  $\F_{q^2}^n$, 
\begin{equation*}
    \rho(U)=
    \left\{
            \begin{array}{cl}
                4n                            &  \textit{if }   u \geq 4, \\
               u^2 & \textit{if}\,\, u \leq 2 ,
            \end{array}
            \right .
        \end{equation*}
and $\rho(U) \in \{4n, 4n-1\}$ if $u=3$.
\end{theorem}
\begin{proof}
Consider the maximum rank distance code $\mathcal{H}_{n,d,s}$ with the parameters as in the statement. By \eqref{rk-f-restricted} and Lemma \ref{dual-d-odd-her}, $\rho(U)$ is determined  for any $u$-dimension $\F_{q^2}$-subspace of $\F_{q^2}^n$ whenever $u \neq 3$.
 Consider the code
  , which is given by:
    \begin{equation*}
       \mathcal{H}'=X^{q^{2s}} \circ \mathcal{H}_{n,n-3,s}= \{ a X + b X^{q^{2s}} + b^{q^s} X^{q^{4s}}+a^{q^{s}}X^{q^{6s}} : a,b \in \F_{q^{2n}}\}.
      \end{equation*}
and note that the polynomials in $\mathcal{H}'$ have $q^{2s}$-degree at most $3$.\\
Moreover, the polymatroid $\mathcal{M}[\mathcal{H}']$ is equivalent to $\mathcal{M}[\mathcal{H}_{n,n-3,s}]$, see \cite[Proposition 6.7]{gorla2019rank}. The result follows by applying an argument similar to that used in the proof of Theorem \ref{alternating-missing}.
\end{proof}


\begin{theorem}
    Let $n,d$ be integers with $n \geq 5$ odd and $n-d=2$ and  consider $\mathcal{M}[\mathcal{E}_{n,d,s}]$. Then, for any $u$-dimensional $\F_{q^2}$-subspace $U$ of  $\F_{q^2}^n$, 
\begin{equation*}
    \rho(U)=
    \left\{
            \begin{array}{cl}
                3n                            &  \textit{if }   u \geq 4, \\
               u^2 & \textit{if}\,\, u \leq 2 ,
            \end{array}
            \right .
        \end{equation*}
and $\rho(U) \in \{3n, 3n-1\}$ if $u=2$.
\end{theorem}
\begin{proof}
  Consider the maximum rank distance $(n-2)$-code $\mathcal{E}_{n,n-2,s}$ with $n\geq 5$ odd. By \eqref{rk-f-restricted} and Lemma \ref{dual-d-odd-her}, $\rho$ is determined if $u \neq 2$. Then, let us assume $u=2$. Consider the following code:
    \begin{equation*}
       \mathcal{E}'=X^{q^{s(n+1)}} \circ \mathcal{E}_{n,n-2,s}= \{ a^q X + b X^{q^{2s}} + a^{q^s} X^{q^{4s}} : a,b \in \F_{q^{2n}}\}.
      \end{equation*}
In order to get the result, it is enough that the polynomials in $\mathcal{E}'$ have $q^s$-degree at most $2$ and by applying an argument similar to that used in the proof of Theorem \ref{alternating-missing}.
\end{proof}

\section{Duality}\label{last_section}

In this last section, we collect a few results on the connection between the dual of a $q$-polymatroid  and the dual of a code in the restricted settings.

Let $E$ be an $\F_q$-vector space and $\M=(\mathscr{L}(E),\rho)$ be a $(q,r)$-polymatroid. For every
$U \in \mathscr{L}(E)$, we define
\begin{equation}\label{dualrank}
\rho^*(U)= r \cdot \dim U - \rho(E) + \rho(U^\perp).
\end{equation}
It is well known that $\M^*=(\mathscr{L}(E),\rho^*)$ is a $(q,r)$-polymatroid as well, called
the \emph{dual} of~$\M$.
In~\cite[Theorem~8.1]{gorla2019rank}, the authors showed that if
$\cC \subset \F_q^{m \times n}$, then
\[
\mathcal{M}^*_c[\cC] =\mathcal{M}_c[\cC^\perp]
\quad\text{and}\quad
\mathcal{M}^*_r[\cC] = \mathcal{M}_r[\cC^\perp],
\]
where $\cC^\perp$ denotes the Delsarte dual of $\cC$.\\

If $\cC \subset \cX_{n,q}$, it is natural to investigate whether there exists a meaningful relation between the rank function of $\mathcal{M}[\cC^*]=(\mathscr{L}(E),\rho_*)$ and that of the $q$-polymatroid $\mathcal{M}^*[\cC]=(\mathscr{L}(E),\rho^*)$.

For this purpose, we establish some preliminary results.

\begin{lemma}\label{inequality}
Let $\cC,\cD \subset \F_{q}^{m \times n}$ be $\F_q$-linear rank-metric codes. Then, for any subspaces $U,V \leq \F_q^m$,
\begin{equation}\label{eq:ineq}
\begin{aligned}
&\dim \cC(U,c)+\dim \cC(V,c)
+\dim (\cC \cap \cD)(U\cap V,c)
+\dim (\cC \cap \cD)(U+V,c)\\
&\leq
\dim \cC(U\cap V,c)+\dim \cC(U+V,c)
+\dim (\cC \cap \cD)(U,c)
+\dim (\cC \cap \cD)(V,c).
\end{aligned}
\end{equation}
\end{lemma}

\begin{proof}
By \eqref{eq:relations}, the left-hand side of \eqref{eq:ineq} can be rewritten as
\begin{equation*}
\begin{aligned}
&\dim\bigl(\cC(U,c)+\cC(V,c)\bigr)
+\dim \cC(U\cap V,c)\\
&\quad
+\dim\bigl((\cC\cap\cD)(U,c)\cap(\cC\cap\cD)(V,c)\bigr)
+\dim (\cC\cap\cD)(U+V,c).
\end{aligned}
\end{equation*}

\noindent Since $\cC(U,c)+\cC(V,c)\subseteq \cC(U+V,c)$ and applying again the Grassmann formula, we obtain that 
\begin{equation*}
\begin{aligned}
&\dim\bigl(\cC(U,c)+\cC(V,c)\bigr)
+\dim \cC(U\cap V,c)\\
&\quad
+\dim\bigl((\cC\cap\cD)(U,c)\cap(\cC\cap\cD)(V,c)\bigr)
+\dim (\cC\cap\cD)(U+V,c) \leq\\ 
&\dim \cC(U+V,c)
+\dim \cC(U\cap V,c) +\dim (\cC\cap\cD)(U,c)\\
&\quad+\dim (\cC\cap\cD)(V,c)
-\dim\bigl((\cC\cap\cD)(U,c)\\
& \quad +(\cC\cap\cD)(V,c)\bigr)+\dim (\cC\cap\cD)(U+V,c).
\end{aligned}
\end{equation*}

Finally, since
\[
\dim\bigl((\cC\cap\cD)(U,c)+(\cC\cap\cD)(V,c)\bigr)
\leq
\dim (\cC\cap\cD)(U+V,c),
\]
the inequality in \eqref{eq:ineq} follows.
\end{proof}

Let $\mathcal{C}, \mathcal{D}$ be $\mathbb{F}_q$-linear codes in $\mathbb{F}_q^{m\times n}$, and let
$\rho_{\mathcal{C}}, \rho_{\mathcal{D}}, \rho_{\mathcal{C}\cap\mathcal{D}}$ be the rank functions of the column $(q,n)$-polymatroids associated with $\mathcal{C}$, $\mathcal{D}$, and $\mathcal{C}\cap\mathcal{D}$, respectively. For the subspaces $U \leq \F_{q}^m$ define the integer-valued functions
\begin{equation}\label{quotient1}
\rho_{\footnotesize{\faktor{\cC}{\cC \cap\cD}}}(U)=\rho_{\cC}(U)-\rho_{\cC \cap \cD}(U)
\end{equation}
and 
\begin{equation}\label{quotient2}
\rho_{\footnotesize{\faktor{\cD}{\cC \cap\cD}}}(U)=\rho_{\cD}(U)-\rho_{\cC \cap \cD}(U)
\end{equation}

Therefore, we have the following.

\begin{theorem}\label{th:quotient}
Let $\mathcal{C}, \mathcal{D}$ be $\mathbb{F}_q$-linear codes in $\mathbb{F}_q^{m\times n}$, and let
$\rho_{\mathcal{C}}, \rho_{\mathcal{D}}, \rho_{\mathcal{C}\cap\mathcal{D}}$ denote the rank functions of the column $(q,n)$-polymatroids associated with $\mathcal{C}$, $\mathcal{D}$, and $\mathcal{C}\cap\mathcal{D}$, respectively.
Then the pairs
\[
\bigl(\mathscr{L}(\F_q^m),\rho_{\faktor{\cC}{\cC\cap\cD}}\bigr)
\quad\text{and}\quad
\bigl(\mathscr{L}(\F_q^m),\rho_{\faktor{\cD}{\cC\cap\cD}}\bigr)
\]
are $(q,n)$-polymatroids. In particular, for every $U\le \F_q^m$,
\begin{equation}\label{expr1}
\begin{aligned}
\rho_{\faktor{\cC}{\cC\cap\cD}}(U)
&=
\dim(\cC+\cD)
-\dim\bigl(\cC(U^\perp,c)+\cD\bigr),\\
\rho_{\faktor{\cD}{\cC\cap\cD}}(U)
&=
\dim(\cC+\cD)
-\dim\bigl(\cC+\cD(U^\perp,c)\bigr).
\end{aligned}
\end{equation}
\end{theorem}

\begin{proof}
We prove the statement for
$\bigl(\mathscr{L}(\F_q^m),\rho_{\faktor{\cC}{\cC\cap\cD}}\bigr)$. A similar argument applies to 
$\bigl(\mathscr{L}(\F_q^m),\rho_{\faktor{\cD}{\cC\cap\cD}}\bigr)$. By definition,
\begin{align*}
\rho_{\faktor{\cC}{\cC\cap\cD}}(U)
&=\rho_{\cC}(U)-\rho_{\cC\cap\cD}(U)\\
&=\dim\cC-\dim(\cC\cap\cD)
-\Bigl(\dim\cC(U^\perp,c)
-\dim(\cC\cap\cD)(U^\perp,c)\Bigr).
\end{align*}
for any $U  \in \mathscr{L}(\F_q^m)$.
Applying the Grassmann formula, we obtain 

\begin{align*} \rho_{\footnotesize{\faktor{\cC}{\cC\cap\cD}}}(U)&=\dim(\cC+\cD) -\dim\cD\\ &- \left (\dim( \cC(U^\perp,c))- \dim\cC (U^\perp, c)- \dim \cD + \dim (\cC(U^\perp,c) + \cD)\right ), \end{align*} which proves both the first equality in \eqref{expr1}  and that $\rho_{\faktor{\cC}{\cC\cap\cD}}(U) \geq 0$. Moreover, 
\[\rho_{\faktor{\cC}{\cC\cap\cD}}(U)
\le \rho_{\cC}(U)
\le n\dim U\]
and, hence, the property (R1) holds.\\
Now let $U,V \in \mathscr{L}(\F_q^m)$ such that $U \subseteq V$. Since $\cC(V^\perp,c) \subseteq \cC(U^\perp,c)$ and by \eqref{expr1}, we have $\rho_{\footnotesize{\faktor{\cC}{\cC\cap\cD}}}(U)\leq \rho_{\footnotesize{\faktor{\cC}{\cC\cap\cD}}}(V)$, proving (R2). Finally let $W,T$ be $\F_q$-subspaces of $\F_q^m$.\\
Finally, by definition of $\rho_{\faktor{\cC}{\cC\cap\cD}}$, the property (R3) is equivalent to
\begin{equation}\label{eq:ineq2}
\begin{aligned}
&\dim\cC(W^\perp,c)+\dim\cC(T^\perp,c)
+\dim(\cC\cap\cD)((W+T)^\perp,c)\\
&\qquad
+\dim(\cC\cap\cD)((W\cap T)^\perp,c)\\
&\le
\dim\cC((W+T)^\perp,c)
+\dim\cC((W\cap T)^\perp,c)\\
&\qquad
+\dim(\cC\cap\cD)(W^\perp,c)
+\dim(\cC\cap\cD)(T^\perp,c).
\end{aligned}
\end{equation}
for any $W,T \in \mathscr{L}(\F_q^m)$.
Since $(W+T)^\perp=W^\perp\cap T^\perp$ and $(W\cap T)^\perp=W^\perp+T^\perp$, the inequality  in \eqref{eq:ineq2} follows by Lemma~\ref{inequality}.
\end{proof}

Clearly, a similar result to Theorem~\ref{th:quotient} can also be proved by defining analogous functions, as in \eqref{quotient1} and \eqref{quotient2}, starting from the row $(q,n)$-polymatroids associated with the codes $\cC$, $\cD$, and $\cC\cap\cD$, respectively.\\

Now, let $\cC \subset \cX_{n,q}$ be an $\F_q$-linear code. By Remark~\ref{remark} and \eqref{dualrank}, we note that if $\cC$ is an alternating or a symmetric code, $\mathcal{M}^*_c[\cC]=\mathcal{M}^*_r[\cC]$. On the other hand, if $\cC \subset \Her_{n,q}$, then $\rho^*_c(U)=\rho^*_r(U^\sigma)$.
Therefore, as done in Section \ref{qpolyrestricted}, for any code $\cC$ in $\cX_{n,q}$, we will consider only the $q$-polymatroid $\mathcal{M}_c^*[\cC]$, which we denote by $\mathcal{M}^*[\cC]=(\mathscr{L}(\F_{q^\ell}^n),\rho^*)$, $\ell \in \{1,2\}$.\\

\begin{lemma}\label{lem:int}
 Let $\cC \subset \cX_{n,q}$ be a $d$-code of $\cX_{n,q}$,  and let $\cC^*$ and $\cC^\perp$ the dual and the Delsarte dual of $\cC$, respectively. If either $\cX_{n,q} \neq  \Alt_{n,q}$ or $\mathrm{char}\,\F_q \neq 2$ and $\cX_{n,q} =  \Alt_{n,q}$, we have 
$$\cC^* = \cC^\perp \cap \cX_{n,q}.$$
\end{lemma} 

\begin{proof}
The result directly follows by the definition of Delsarte dual as given in \eqref{def:delsarte_dual} and by Definition \ref{def:dual}. On the other hand, if $\cX_{n,q}=\Alt_{n,q}$ and $p \neq 2$, then the result follows by observing that $\underset{1 \leq i < j \leq n}{\Sigma} a_{ij}b_{ij}= -\frac{\Tr(AB^t)}{2}$ for any $A=(a_{ij})$ and $ B=(b_{ij})$ belonging to $ \Alt_{n,q}$.
\end{proof}

Combining Theorem \ref{th:quotient} and Lemma \ref{lem:int}, we derive the following.

\begin{theorem}
 Let $\cC \subset \cX_{n,q}$ be a $d$-code of $\cX_{n,q}$,  and let $\cC^*$ and $\cC^\perp$ be the dual and the Delsarte dual of $\cC$, respectively. If $\cX_{n,q} \neq  \Alt_{n,q}$ or $q$ odd and $\cX_{n,q} =  \Alt_{n,q}$, then for any subspace $U \in \mathscr{L}(E)$, we have 
\begin{equation}\label{*down-*up}
   \rho_{\footnotesize{\faktor{\cC^\perp}{\cC^*}}}(U)=\rho^*(U)-\rho_*(U)=(\dim(\cC^\perp+ \cX_{n,q})-\dim(\cC^\perp(U^\perp,c)+\cX_{n,q})).  
\end{equation}
\end{theorem}

\begin{proof}
 Let $\mathcal{M}[\cC^*]=(\mathscr{L}(E),\rho_*)$ be the $(q^\ell,n)$-polymatroid of $\cC^*$ and $\mathcal{M}^*[\cC]=(\mathscr{L}(E),\rho^*_c)$ be the dual $(q^\ell,n)$-polymatroid of $\cC$, $\ell \in \{1,2\}$.  
By Theorem \ref{th:quotient} and Lemma \ref{lem:int}, we have that
\begin{equation*}
\begin{aligned}
\rho_{\footnotesize{\faktor{\cC^\perp}{\cC^*}}}(U)=\rho^*(U)-\rho_*(U)= \bigl( \dim(\mathcal{C}^\perp+\cX_{n,q})
- \dim\big(\mathcal{C}^\perp(U^\perp,c)+\cX_{n,q}\big) \bigr),
\end{aligned}
\end{equation*}
for any $U \in \mathscr{L}(E)$ and hence the result follows.
\end{proof}

As consequence of the result above, in general if $\cC \subset \cX_{n,q}$, then $\mathcal{M}^*[\cC]$ and $\mathcal{M}[\cC^*]$ are different $(q^\ell,n)$-polymatroids, $\ell \in \{1,2\}$. 

\begin{corollary}
   Let $\cC \subset \Sym_{n,q}$ be an $\F_q$-linear code. If $q$ is even and $\Alt_{n,q} \subset \cC$, then 
   $\mathcal{M}^*[\cC]=\mathcal{M}[\cC^*]$.
\end{corollary}
\begin{proof}
The result follows from the fact that, when the characteristic of the field is $2$, $\Alt_{n,q} = \Sym_{n,q}^{\perp} \subseteq \cC^{\perp} \subseteq \Sym_{n,q}$ and hence, by \eqref{*down-*up}, 
$$\dim(\cC^{\perp} + \Sym_{n,q})
= \dim(\cC^{\perp}(U^{\perp},c) + \Sym_{n,q}).$$
\end{proof}

\bigskip
\bibliographystyle{abbrv}
\bibliography{references.bib}


\vspace{1cm}

\noindent Giovanni Longobardi, Rocco Trombetti,\\
Dipartimento di Matematica e Applicazioni “Renato Caccioppoli”,\\
Università degli Studi di Napoli Federico II,\\
via Cintia, Monte S. Angelo I-80126 Napoli, Italy.\\
email: \texttt{\{giovanni.longobardi, rtrombet\}@unina.it}

\bigskip
\bigskip

\noindent Eimear Byrne,\\
School of Mathematics and Statistics, \\
Science Centre, University College Dublin,\\
Belfield Dublin 4, Ireland.\\
email: \texttt{ebyrne@ucd.ie}

\end{document}